\documentclass[11pt,reqno]{amsart}
\usepackage{amsfonts,amsmath}
\usepackage{amssymb}
\usepackage{textcomp}
\usepackage{enumitem}
\usepackage{graphicx,color}
\numberwithin{equation}{section}
\newcommand{\nc}{\newcommand}
\nc{\parent}[1]{$[\![#1]\!]$}
\newtheorem{theorem}{Theorem}[section]
\newtheorem{lemma}{Lemma}[section]
\newtheorem{example}{Example}[section]
\newtheorem{corollary}{Corollary}[section]
\newtheorem{proposition}{Proposition}[section]
\newtheorem{remark}{Remark}[section]
\newtheorem{definition}{Definition}[section]
\newtheorem{assumption}{Assumption}[section]

\newenvironment{pf-main}{{\sc Proof of Theorem \ref{mainresult}.}\hspace{3mm}}{\qed}
\DeclareMathOperator{\sgn}{sgn}
\DeclareMathOperator{\interior}{int}
\nc{\cadlag}{c\`{a}dl\`{a}g } \nc{\ba}{\begin{array}}
\nc{\ea}{\end{array}} \nc{\be}{\begin{equation}}
\nc{\ee}{\end{equation}} \nc{\bea}{\begin{eqnarray}}
\nc{\eea}{\end{eqnarray}} \nc{\bean}{\begin{eqnarray*}}
\nc{\eean}{\end{eqnarray*}} \nc{\bu}{\bullet} \nc{\nn}{\nonumber}
\nc{\cA}{{\mathcal A}} \nc{\cB}{{\mathcal B}} \nc{\cC}{{\mathcal
C}} \nc{\cD}{{\mathcal D}} \nc{\bbD}{\mathbb{D}}
\nc{\cG}{{\mathcal G}} \nc{\cF}{{\mathcal F}} \nc{\cS}{{\mathcal
S}} \nc{\cU}{{\mathcal U}} \nc{\cH}{{\mathcal H}}
\nc{\cK}{{\mathcal K}}\nc{\cL}{{\mathcal L}}  \nc{\cM}{{\mathcal
M}} \nc{\cO}{{\mathcal O}} \nc{\cP}{{\mathcal P}}
\nc{\bbE}{\mathbb{E}} \nc{\tbA}{\tilde{\bbA}}\nc{\bbA}{\mathbb{A}}\nc{\bbF}{\mathbb{F}}
\nc{\bbEQ}{\mathbb{E}_{\mathbb{Q}}} \nc{\eps}{\varepsilon}
\nc{\bbEP}{\mathbb{E}_{\mathbb{P}}}\nc{\bbL}{\mathbb{L}}
\nc{\what}{\widehat} \nc{\bbP}{\mathbb{P}} \nc{\bbQ}{\mathbb{Q}}
\nc{\del}{\partial} \nc{\Om}{\Omega} \nc{\om}{\omega}
\nc{\bbR}{\mathbb{R}} \nc{\bbN}{\mathbb{N}} \nc{\fps}{$(\Om, \cF,
(\cF_t)_{t\geq 0}, \bbP)$} \nc{\bbC}{\mathbb{C}}
\nc{\bfr}{\begin{flushright}} \nc{\efr}{\end{flushright}}
\nc{\dXt}{\Delta X_{t}} \nc{\dXs}{\Delta X_{s}}
\nc{\bs}{\blacksquare} \nc{\dX}{\Delta X} \nc{\dY}{\Delta Y}
\nc{\dnkx}{\left(X(T^{n}_{k})-X(T^{n}_{k-1})\right)}
\nc{\esssup}{\mathrm{ess}\mbox{ }\mathrm{sup}}
\nc{\essinf}{\mathrm{ess}\mbox{ } \mathrm{inf}}
\nc{\dhats}{\widehat{\delta_s}} \nc{\half} {\frac{1}{2}}

\nc{\ol}{\overline}
\def\rar{\rightarrow}

\nc{\chf}{\mbox{$\mathbf1$}}
\begin{document}

\title{Diffusion transformations, Black-Scholes equation and optimal stopping}
\author{Umut \c{C}etin}
\address{Department of Statistics, London School of Economics and Political Science, 10 Houghton st, London, WC2A 2AE, UK}
\email{u.cetin@lse.ac.uk}
\date{\today}
\begin{abstract}
We develop a new class of path transformations for one-dimensional diffusions that are tailored to alter their long-run behaviour from transient to recurrent or vice versa. This immediately leads to a formula for the  distribution of the first exit times of diffusions, which is recently characterised by Karatzas and Ruf \cite{KR} as the minimal solution of an  appropriate Cauchy problem under more stringent conditions. A particular limit of these transformations also turn out to be instrumental in characterising the stochastic solutions of Cauchy problems defined by the generators of strict local martingales, which are well-known for not having unique solutions even when  one restricts solutions  to have  linear growth. Using an appropriate diffusion transformation we show that the aforementioned stochastic solution can be written in terms of the unique classical  solution of an {\em alternative} Cauchy problem with suitable boundary conditions. This in particular  resolves the long-standing issue of non-uniqueness with the Black-Scholes equations in derivative pricing in the presence of {\em bubbles}.  Finally, we use these path transformations to propose a unified framework for solving explicitly the optimal stopping problem for one-dimensional diffusions with discounting, which in particular is relevant for the pricing and the computation of optimal exercise boundaries of perpetual American options.
\end{abstract}
\maketitle

\section{Introduction} 
 Conditioning the paths of a given Markov process $X$ to stay in a certain subset of the path space is a well-studied subject which has become synonymous with the term $h$-transform.  If one wants to condition the paths of $X$ to stay in a certain set, the classical recipe consists of finding an appropriate excessive function $h$,  defining the transition probabilities of the conditioned process via  $h$, and constructing on the canonical space a Markov process $X^h$ with these new transition probabilities. This procedure is called an $h$-transform. In particular if $h$ is a minimal excessive function with a pole at $y$ (see Section 11.4 of \cite{ChungWalsh} for definitions), then $X^h$ is the process $X$ conditioned to converge to $y$ and {\em killed} at its last exit from $y$. We  refer the reader to Chapter 11 of \cite{ChungWalsh} for an in-depth analysis of $h$-transforms.
 
  This paper proposes a new class of path transformations for  one-dimensional regular diffusions with stochastic differential equation (SDE) representation. The new transformations are aimed at switching the behaviour of the diffusion  from transient to recurrent or vice versa.   We introduce the concept of  {\em recurrent transformation} 
 in Section \ref{s:rectr} and characterise these transforms via weak solutions of SDEs. Roughly speaking, a recurrent transformation adds a drift term to the original SDE of $X$ so that the resulting process is a {\em recurrent} regular  diffusion with the same state space whose law is locally absolutely continuous with respect to the original law.  Although the recurrent transformation is  at first sight meaningful only for transient diffusions, we note a special class of recurrent transformations  in Theorem \ref{t:prctr} that is applicable not only to transient diffusions but also to recurrent ones. This transform, by adding again a certain drift, results in a {\em positively recurrent} diffusion. For example, this transformation turns a standard Brownian motion to a Brownian motion with alternating drift, which appears in the studies of the bang-bang control problem (see Example \ref{ex:BMalt}). 

As a first application of the recurrent transformation, we compute in   Corollary \ref{c:survfunction} the distribution of the  first exit time from an interval for a given diffusion.  Although the formula does not provide an expression in closed form in general,  a simple Monte Carlo algorithm will provide a sufficiently close estimate. 

The distribution of first exit times has attracted the attention of researchers working on problems arising in the Monte Carlo simulation of stochastic processes (see, e.g., \cite{baldi1995}, \cite{BCI}, \cite{GobetKilled}, \cite{GMkilled}, and the references therein). Yet precise formulas for the distribution of  exit times of diffusions  have rarely been the subject of a thorough investigation. The recent paper of Karatzas and Ruf \cite{KR} seems to be the only work in the literature that addresses this problem in the general framework of one-dimensional diffusions. With an additional assumption on the local H\"older continuity of the coefficients of the SDE satisfied by $X$ they have shown that the distribution function of the first exit time was the minimal nonnegative solution of a particular Cauchy problem. Although this is a useful characterisation from a theoretical perspective, finding the smallest solution of a Cauchy problem is in general not a feasible numerical task. Our formula in Corollary \ref{c:survfunction} thus provides a way of computing the minimal solutions of the class of Cauchy problems considered by Karatzas and Ruf.

As described briefly in Remark \ref{r:simulation} recurrent transformations can also  be used to improve the accuracy of discrete Euler approximations of a diffusion killed when exiting a bounded interval. As shown by Gobet \cite{GobetKilled} the discretisation error for such Euler schemes is of order $N^{-\half}$, where $N$ is the number of discretisations, as opposed to $N^{-1}$, which is the rate of convergence for discrete Euler schemes for diffusions without killing. As the recurrent transformation removes the killing by passing to a locally absolutely continuous probability measure, it can be used to bring the convergence rate back to $N^{-1}$ using the recipe in Remark \ref{r:simulation}. This important application of recurrent transformations will be  studied rigorously in a subsequent paper.

Section \ref{s:limit} is devoted to the convergence of  certain recurrent transforms when $X$ is nonnegative and on natural scale. Under a mild condition on the diffusion coefficient of $X$ we show that a particular sequence of recurrent transformations converges monotonically to the $h$-transform of $X$, where $h(x)=x$. We observe that the nature of this convergence depends crucially on whether $X$ is a strict local martingale or not. In particular,  we construct on a single probability space a sequence of recurrent transforms that increases a.s. to a diffusion that has the same law as the aforementioned $h$-transform. The limiting diffusion is non-exploding on $[0,\infty)$ if and only if $X$ is a true martingale. 

Our interest in local martingales in fact stems from the financial models with bubbles. If a financial model admits no arbitrage opportunities,  the discounted stock price $X$ must follow a nonnegative local martingale under a so-called risk-neutral measure by the Fundamental Theorem of Asset Pricing \cite{DS}. When $X$ is not a martingale but a strict local martingale, the stock price exhibits a bubble and many results in the arbitrage pricing theory become invalid (see \cite{CH} and \cite{PP} for some examples). One particular issue concerns the Black-Scholes pricing equation for a European option that pays the amount of $g(X_T)$ to its holder at time $T$ for some $g \geq 0$. The arbitrage pricing theory suggests that the fair price of this option at time $t$ is $v(T-t, X_t)$, where $v(t,x):=E^x[g(X_t)]$. Under mild conditions on $X$ and a continuity and linear growth assumption on $g$, Ekstr\"om and Tysk \cite{ET} have shown that $v$  satisfies the Cauchy problem
\be \label{e:cauchyi}
u_t=Au, \qquad u(0,\cdot)=g,
\ee
where $A$ is the infinitesimal generator of $X$. As a consequence, Ekstr\"om and Tysk have  observed in \cite{ET} that (\ref{e:cauchyi}) admits multiple nonnegative solutions when $X$ is a strict local martingale and $g(x)=x$. Namely, they have identified $u(t,x):=x$ and $v(t,x)=E^x[X_t]$ as such two distinct  solutions. Note that $X$ being a strict local martingale implies $v(t,x)=E^x[X_t]<x=u(t,x)$. Thus, $x-E^x[X_t]$ is a solution of (\ref{e:cauchyi}) when $g\equiv 0$.  However, this immediately leads to the conclusion that there are infinitely many solutions of at most linear growth to (\ref{e:cauchyi}) whenever $g$ is of at most linear growth. Indeed, by the above discussion  for any $\alpha >0$, $\tilde{u}(t,x):=E^x[g(X_t)] + \alpha (x-E^x[X_t])$ is a nonnegative solution of (\ref{e:cauchyi}) when $g$ is of at most linear growth. Moreover, Ekstr\"om and Tysk have also  shown that $E^x[g(X_t)]$ is of at most linear growth when $g$ is continuous function of at most linear growth.  This in turn renders $\tilde{u}$  of linear growth.  Hence, restricting solutions to have at most linear growth does not yield uniqueness for the above Cauchy problem.  

 Bayraktar and Xing \cite{BX} have followed up this question by showing that the uniqueness of the Cauchy problem is determined by the martingale property of $X$.  Later, Bayraktar et al. \cite{BKX} have extended the scope of these conclusions to Markovian stochastic volatility models. 

The absence of uniqueness for solutions of (\ref{e:cauchyi}) is  especially problematic if one wants to compute the option prices by solving (\ref{e:cauchyi}) numerically.    Also note that one will also fail to compute $E^x[g(X_t)]$ using a Monte-Carlo simulation when $g$ is of linear growth and $X$ is a strict local martingale. Indeed, if, e.g.,  $g(x)=x$, the Monte-Carlo algorithm will yield $x$ for $E^x[X_t]$ since the discretisation of $X$ via the Monte-Carlo scheme will result in a {\em true} martingale for the approximating process. To resolve this issue  we establish in Section \ref{s:BS} a new characterisation of $E^x[g(X_t)]$   in terms of the unique solution of an alternative Cauchy problem. We show that the function $(t,x) \mapsto E^x[g(X_t)]$, after an appropriate scaling, becomes the unique solution of 
\[
w_t= \tilde{A}w
\]
with certain initial and boundary conditions, when $\tilde{A}$ is the generator of a suitable $h$-transform of $X$. More precisely, this $h$-transform coincides with the  one that is obtained as the limit of recurrent transforms in Section \ref{s:limit}. One interesting corollary of the main result of this section  is that for any $t>0$ the valuation function $E^x[g(X_t)$ is of strictly sublinear growth at $\infty$ when $g$ is of at most linear growth and the stock price is given by a strict local martingale. In particular, $\lim_{x \rar \infty}\frac{E^x[X_t]}{x}=0$ for any $t>0$ if $X$ is a strict local martingale.

While Section \ref{s:BS} is on the valuation of European options, Section \ref{s:OS} considers the pricing of perpetual American options. In order to price such an option with payoff $g$, one needs to solve the optimal stopping problem
\[
\sup_{\tau \leq \zeta}E^x\left[e^{-\lambda \tau}g(X_{\tau})\right],
\]
where $\zeta$ is the (possibly finite) lifetime of the diffusion $X$ and the {\em discount rate} $\lambda>0$ corresponds to the constant interest rate.

Peskir and Shiryaev \cite{PS} give an excellent survey of available methods to tackle this problem. One approach to the above consists of solving  a free boundary problem associated to the infinitesimal generator of $X$. Another approach is via the characterisation of $\lambda$-excessive functions of $X$ as the value function for the optimal stopping problem is the least $\lambda$-excessive majorant of $g$. This is the path taken by Dayanik and Karatzas in \cite{DK}.  Beibel and Lerche \cite{BL} have also proposed a new methodology based on simple martingale arguments, which can also be interpreted as change of measure arguments as observed by \cite{LU}. While the approach based on the solution of a free boundary problem rarely provides explicit solutions, the other two have the potential to offer explicit or semi-explicit solutions. However, these solutions  crucially depends on the assumption that one has the solutions of a family of Sturm-Liouville equations at hand. Moreover, the solution techniques offered in \cite{DK} and \cite{BL} differ for different boundary behaviour exhibited by $X$,  i.e. whether the boundaries  of the state space of $X$ are absorbing or natural, etc. Furthermore, how the function $g$ behaves near the boundaries also matters. For instance, Beibel and Lerche \cite{BL} have to check five conditions on the behaviour of $g$ to determine the solution. It is also worth to note the recent work of  Lamberton and Zervos \cite{LZ} who analyse a large class of optimal stopping problems via variational equalities defined by the generator of $X$ and $g$ without the assumption that $g$ is continuous. 

Section \ref{s:OS} presents a unified solution to the above optimal stopping problem that does not vary depending on the behaviour of $g$ or $X$ near the boundaries. We use the specific recurrent transform of Proposition \ref{p:AOtr}, which is applicable to transient as well as recurrent diffusions, to determine whether the value function is finite. We show that  the value function is finite if and only if  $g$ satisfies the single condition (\ref{e:nc0}), which depends only on the knowledge of $u^{\lambda}(\cdot, y)$, the $\lambda$-potential density, for some $y$. This recurrent transform also changes the optimal stopping problem to one without discounting. However, the new problem becomes two-dimensional. In order to reduce the dimension of the problem to one, we apply the transformation that is defined in Section \ref{s:another}, which is aimed at conditioning the recurrent transformation to have a certain behaviour at the boundary points and become transient. After this transformation all that remains to do is to solve 
\[
\sup_{\tau}\tilde{E}^x[\bar{g}(X_{\tau})],
\]
where $\bar{g}$ is a function that depends only on $g$ and $u^{\lambda}(\cdot,y)$, and $\tilde{E}$ corresponds to the expectation operator with respect to the law of the diffusion after the final transformation. Solution to the above is easy and well-known since Dynkin \cite{Dyn}: After a change of scale, the value function of the above optimal stopping problem is the smallest concave majorant of $\bar{g}$. 

It has to be noted that Cisse et al. \cite{CPT} have attacked this problem using $h$-transforms. However, as we explain in detail in Remark \ref{r:cetal} the authors make some implicit assumptions regarding the boundary behaviour of $X$ as well as the function $g$ in the proof of their key arguments. These assumptions in particular exclude the diffusion processes with infinite lifetime. As we mentioned above, our approach is general and do not impose any conditions on $X$ other than the regularity and the Engelbert-Schmidt conditions that ensures an SDE representation for $X$.

In essence our framework is fundamentally different in spirit from \cite{CPT} and \cite{BL} in the sense that it gives a probabilistic interpretation of the value function and the optimal stopping boundaries under a locally absolutely continuous measure in the classical framework of Dynkin \cite{Dyn} with no discounting. The works of \cite{CPT} and \cite{BL}, on the other hand, obtain the solution by a clever algorithm of maximisation provided one has the solutions of a family of Sturm-Liouville equations.

Differently from our treatment in Section \ref{s:BS} we do not investigate the impact of martingale property of $X$ on the valuation of perpetual American options as the methodology is the same for the martingales as well as the local martingales. We refer the reader to \cite{BKX2} for a thorough analysis of the influence of the martingale property in a general framework. 

An outline of this paper is as follows. Section \ref{s:prelim} gives a brief overview of several concepts related to one-dimensional diffusions that will be used throughout the paper. Section \ref{s:rectr} introduces the concept of recurrent transformations while Section \ref{s:limit} considers their limit in relation to the local martingale property of $X$. Section \ref{s:another} defines a transform designed specifically  for recurrent diffusions that is different than the typical $h$-transform but  will still render them transient, which will be useful in Section \ref{s:OS}.  Section \ref{s:BS} provides a resolution to the non-uniqueness issue of the Black-Scholes pricing equation and Section \ref{s:OS} addresses the optimal stopping problem. Section \ref{s:conc} concludes.  Proofs of certain results that are not contained in the main body is included in the Appendix. 

{\bf Acknowledgements:} I'd like to thank Johannes Ruf for the useful discussions  and the anonymous referees for their comments that led to several improvements.
\section{Preliminaries} \label{s:prelim}
Let $X$ be a regular  diffusion on $(l,r)$, where $ -\infty \leq l <r \leq \infty$.  We assume that if any of the boundaries are reached in finite time, the process is absorbed at that boundary. This is the only instance when the process can be `killed', we do not allow killing inside $(l,r)$. Such a diffusion is uniquely characterised by its scale function $s$ and speed measure $m$, defined on the Borel subsets of the open interval $(l,r)$. The set of points that can be reached in finite time starting from the interior of $(l,r)$ and the entrance boundaries will be denoted by $I$. That is, $I$ is the union of $(l,r)$ with the regular, exit or entrance boundaries.  The law induced on $C(\bbR_+,I)$,  the space of $I$-valued continuous functions on $[0,\infty)$, by $X$ with $X_0=x$ will be denoted by $P^x$ as usual, while $\zeta$ will correspond to its lifetime, i.e. $\zeta :=\inf\{t>0:X_{t} \in \{l,r\}\}$. For concreteness we assume that  $X$ is the coordinate process on the canonical space $\Om:=C(\bbR_+,I)$, i.e. $X_t(\om)=\om(t)$ for all $t\geq 0$. However,  this assumption is only for convenience and one can work with other measurable spaces as long as the measures $(P^x)_{x \in I}$ are properly defined. The filtration $(\cF_t)_{t \geq 0}$ will correspond to the universal completion of the natural filtration of $X$ and, therefore, is right continuous since $X$ is strong Markov by definition (see Theorem 4 in Section 2.3 in \cite{ChungWalsh}). We will also set $\cF:=\bigvee_{t \geq 0} \cF_t$. If $\mu$ is a measure on some open interval $(a,b)$ and $f$ is a nonnegative or $\mu$-integrable measurable function, the integral of $f$ with respect to $\mu$ will be denoted by $\int_{(a,b)}f(x)\mu(dx)$ unless $\mu$ is absolutely continuous with respect to the Lebesgue measure $dx$,  in which case we shall write $\int_a^b f(x)\mu(dx)$.

In what follows we will often replace $\zeta$ with $\infty$ when dealing with the limit values of the processes as long as no confusion arises.  Recall that in terms of the first hitting times,   $T_y:=\inf\{t> 0: X_t=y\}$ for $y \in (l,r)$, the regularity amounts to $P^x(T_y<\infty)>0$ whenever $x$ and $y$ belongs to the open interval $(l,r)$. This assumption entails in particular that $s$ is strictly increasing and continuous (see Proposition VII.3.2 in \cite{RY}) and $0<m((a,z))<\infty$ for all $l<a<z<r$ (see Theorem VII.3.6 and the preceding discussion in \cite{RY}). 

Recurrence or transience of $X$ depends on the behaviour of $s$ near the boundary points. More precisely, $X$ is transient if and only if at least one of $s(l)$ and $s(r)$ is finite. Since $s$ is unique only up to an affine transformation, we will use the following convention throughout the text: 
\begin{itemize}
	\item $s(l)=0$ whenever finite,
	\item  $s(r)=1$ whenever finite.
\end{itemize} 
Note that in view of our foregoing assumptions one can easily deduce that $X_{\zeta -}\in \{l,r\}$ when $X$ is transient. We refer the reader to \cite{BorSal} for a summary of results and references on one-dimensional diffusions. The definitive treatment of such diffusions is, of course, contained in \cite{IM}. The recent manuscript of Evans and Hening \cite{EH} contains a detailed discussion with proofs of some aspects of the potential theory of one-dimensional diffusions.
\begin{remark} It has to be noted that notion of recurrence that we consider here excludes some recurrent solutions of one-dimensional SDEs with time-homogeneous coefficients since we kill our diffusion as soon as it reaches a regular boundary point. A notable example is a squared Bessel process with  dimension $\delta <2$, which solves the following SDE:
\[
X_t= x+2 \int_0^t\sqrt{X_s}dB_s +\delta t.
\]
The above SDE has a global strong solution, i.e. solution for all $t \geq 0$, which is recurrent (see Section XI.1 of \cite{RY}). However, the point $0$ is reached a.s. and is instantaneously reflecting by Proposition XI.1.5 in  \cite{RY}. As such, it violates our assumption of a diffusion being killed at a regular boundary. According to  our assumption, a squared Bessel process of dimension $0<\delta<2$ has to be killed as soon as it reaches $0$ and, thus, is a {\em transient} diffusion. 
\end{remark}
As our focus is on diffusions that are also solutions of SDEs, we further impose the so-called {\em Engelbert-Schmidt conditions}. That is, we shall assume the existence of measurable functions $\sigma:(l,r)\to \bbR$ and $b: (l,r)\to \bbR$ such that 
\be \label{e:ESR}
\sigma(x) >0  \mbox{ and } \exists \eps >0 \mbox{ s.t. } \int_{x-\eps}^{x+\eps}\frac{1+ |b(y)|}{\sigma^2(y)}dy <\infty \mbox{ for any $x \in (l,r)$.}
\ee
Under this assumption (see \cite{ES} or Theorem 5.5.15 in \cite{KS}) there exists a {\em unique} weak solution (up to the exit time from the interval $(l,r)$) to the SDE
\be \label{e:sdeX}
X_t=x+\int_0^t\sigma(X_s)dB_s + \int_0^t b(X_s)ds, \qquad t < \zeta,
\ee
where $\zeta=\inf\{t> 0: X_{t}\in \{l,r\}\}$ and $l<x<r$.  
Moreover,  condition (\ref{e:ESR}) further implies one can take
\be \label{e:scalespeed}
s(x)=\int_C^x \exp\left(-2 \int_c^z \frac{b(u)}{\sigma^2(u)}du\right)dz
\; \mbox{ and }\;  m(dx)=\frac{2}{s'(x)\sigma^2(x)}dx, \mbox{ for some } (c,C) \in (l,r)^2.
\ee
We collect the assumptions on $X$ in the following:
\begin{assumption} \label{a:reg} $X$ is a regular one-dimensional diffusion on $(l,r)$ such that
\[
X_t=X_0+ \int_0^t\sigma(X_s)dB_s + \int_0^t b(X_s)ds, \qquad t < \zeta,
\]
where   $\sigma:(l,r)\to \bbR$ and $b: (l,r)\to \bbR$ satisfy (\ref{e:ESR}), $\zeta=\inf\{t>0: X_{t}\in \{l,r\}\}$.
\end{assumption}

In the sequel the {\em extended generator} of $X$ will be denoted by $\bbA$. Following Definition VII.1.8 of Revuz and Yor \cite{RY} we will write $g=\bbA f$ for a given Borel measurable function $f$, if there exists Borel function $g$ such that, for each $x \in I$, i) $P^x$-a.s. $\int_0^t|g(X_s)|ds<\infty$ for every $t>0$, and ii)
\[
f(X_t)-f(X_0)-\int_0^t g(X_s)ds
\] 
is $P^x$-local martingale. In this case $f$ is said to be in the domain of $\bbA$. If $f$ is $C^2$ on $I$, then $\bbA$ becomes a second order differential operator, i.e.
\[
\bbA f(x)=\half\sigma^2(x)f''(x)+b(x)f'(x).
\]

 Any regular transient diffusion on $(l,r)$ has a finite potential density, $u: (l,r)^2 \to \bbR_+$, with respect to its speed measure (see Paragraph 11 in Section II.1 of \cite{BorSal}). That is, for any nonnegative and measurable  $f$ vanishing at accessible boundaries
\[
Uf(x):=\int_0^{\infty}E^x[f(X_t)]dt=\int_l^r f(y)u(x,y)m(dy).
\]
The above implies that the potential density can be written in terms of the transition density\footnote{For the existence of this transition density and its boundary behaviour see Mc Kean \cite{MKpar}. }, $(p(t,\cdot,\cdot))_{t \geq 0}$, of $X$ with respect to its speed measure:
\[
u(x,y)=\int_0^{\infty}p(t,x,y)dt.
\]
The above in particular implies that $u(x,y)=u(y,x)$ since $p(t,\cdot,\cdot)$ is symmetric for each $t>0$ (see p. 520 of \cite{MKpar}). If $X$ is recurrent , either  $Uf\equiv \infty$ or $Uf\equiv 0$ (see Theorem 1 in Section 3.7 of \cite{ChungWalsh}). Therefore, potential density only makes sense for transient diffusions.

We will denote by $({L}_t^x)_{x \in (l,r)}$  the family of {\em semimartingale} local times\footnote{Observe that the {\em diffusion} local time, $\tilde{L}$, in Paragraph 13 in Section II.2 of \cite{BorSal} is defined via $\int_0^t f(X_s)ds =\int_{l}^{r}f(x)\tilde{L}^x_t m(dx)$. Comparing this with the occupation times formula for the semimartingale local time reveals the relationship $\frac{2}{s'(x)}\tilde{L}^x=L^x$.} associated to $X$. Recall that the occupation times formula for the semimartingale local time is given by
\[
\int_0^t f(X_s)\sigma^2(X_s)ds =\int_{l}^{r}f(x)L^x_t dx.
\]

In the case of one-dimensional transient diffusions  the  distribution of $L^y_{\infty}$ is known explicitly in terms of the potential density (see p.21 of \cite{BorSal}). In particular, 
\be \label{e:lawLT}
P^y(L^y_{\infty}>t)=\exp\left(-\frac{s'(y)t}{2u(y,y)}\right).
\ee
%

Note that if $s(l)=0=1-s(r)$, then $P^x(X_{\infty}=r)=s(x)=1-P^x(X_{\infty}=l)$ and
\be \label{e:psiff}
P^x(T_y<\infty)= \left\{\ba{rl}
\frac{s(x)}{s(y)}, & y\geq x; \\
\frac{1-s(x)}{1-s(y)}, & y < x.
\ea \right . \qquad \qquad u(x,y)= s(x)(1-s(y)), \; x \leq y.
\ee
On the other hand, if $s(l)=0$ and $s(r)=\infty$, then $X_t \rar l$, $P^x$-a.s. for any $x \in (l,r)$, which in turn implies
\be \label{e:psifi}
 P^x(T_y<\infty))= \left\{\ba{rl}
\frac{s(x)}{s(y)}, & y\geq x; \\
1, & y < x.
\ea \right . \qquad \qquad u(x,y)= s(x), \; x \leq y.
\ee
Similarly, if $s(l)=-\infty$ and $s(r)=1$, then $X_t \rar r$, $P^x$-a.s. for any $x \in (l,r)$, and 
\be \label{e:psiif}
P^x(T_y<\infty)= \left\{\ba{rl}
1, & y\geq x; \\
\frac{1-s(x)}{1-s(y)}, & y < x.
\ea \right . \qquad \qquad u(x,y)= 1-s(y), \; x \leq y.
\ee
While the potential density is finite only for transient diffusions, one can define a so-called {\em $\alpha$-potential density} that exists and is finite for all diffusions for all $\alpha>0$. For any nonnegative and measurable function $f$ vanishing at accessible boundaries, one defines
\[
U^{\alpha}f(x):=\int_0^{\infty}e^{-\alpha t} E^x[f(X_t)]dt.
\]
Thus, if we let
\[
u^{\alpha}(x,y):=\int_0^{\infty}e^{-\alpha t} p(t,x,y)dt,
\]
we obtain
\[
U^{\alpha}f(x)=\int_l^r f(y)u^{\alpha}(x,y)m(dy).
\]
$u^{\alpha}(\cdot,\cdot)$ is called the $\alpha$-potential density and is symmetric in  $(l,r)^2$ for all $\alpha>0$. An alternative and very useful expression for $u^\alpha$ is given in terms of the fundamental solutions of the equation $\bbA f =\alpha f$. That is,
\be \label{e:ualpha}
u^{\alpha}(x,y)=\frac{\psi_{\alpha}(x)\phi_{\alpha}(y)}{w_{\alpha}}, \qquad x \leq y,
\ee
where $\psi_{\alpha}$ and $\phi_{\alpha}$ are, respectively, the increasing and decreasing nonnegative solutions of $\bbA f=\alpha f$ subject to certain boundary conditions (see p.19 of \cite{BorSal}), and $w_{\alpha} $ is the Wronskian given by
\[
w_{\alpha}=\frac{\psi'_{\alpha}(x)\phi_{\alpha}(x)-\psi_{\alpha}(x)\phi'_{\alpha}(x)}{s'(x)},
\]
which is independent of $x$. Consequently, using the relationship between the fundamental solutions of $\bbA f=\alpha f$ and the Laplace transforms of hitting times (see p.18 of \cite{BorSal}), we have
\be \label{e:LThittime}
E^x\left[\exp\left(-\alpha T_y\right)\right]=\frac{u^{\alpha}(x,y)}{u^{\alpha}(y,y)}.
\ee
We refer the reader to Chap.~II of Borodin and Salminen \cite{BorSal} for a summary of results concerning one-dimensional diffusions including the ones sketched above. 
\section{Recurrent transformations of  diffusions} \label{s:rectr}
This section introduces a new kind of path transformation for regular diffusions that produces a recurrent diffusion whose law is locally  absolutely continuous with respect to that of the original diffusion.  To wit, suppose  $h$ is a non-negative $C^2$-function and $M$ an adapted continuous process of finite variation so that $h(X)M$ is a non-negative local martingale. If $(\tau_n)$ is a localising sequence for this local martingale, using Girsanov's theorem we arrive at a weak solution on $[0,\tau_n]$ to the following SDE for any given $x \in (l,r)$:
\be \label{e:sderectr}
X_t=x+ \int_0^t \sigma(X_s)dB_s +\int_0^t \left\{b(X_s) + \sigma^2(X_s)\frac{h'(X_s)}{h(X_s)}\right\}ds.
\ee
We can associate to the above SDE the scale function 
\be \label{e:recscale}
s_h(x):=\int_c^x \frac{s'(y)}{h^2(y)}dy, \qquad x \in (l,r),
\ee
provided that the integral is finite for all $x \in (l,r)$, which in particular requires $h>0$ on $(l,r)$. What we would like to achieve is to extend this procedure by taking $n \rar \infty$ and obtain a recurrent diffusion. The latter will require  $-s_h(l+)=s_h(r-)=\infty$. We shall see in this section that this property alone is sufficient to obtain a {\em recurrent} weak solution of (\ref{e:sderectr}) on $[0,\infty)$ under some mild conditions on $h$.

Using $h$ and $M$ to get a recurrent process imposes some boundary conditions on $h$. Indeed, if $s(l)=0$ (resp. $s(r)=1$), in order to have $s_h(l+)=-\infty$ (resp. $s_h(r-)=\infty$), we must have $\lim_{x \rar l}h(x)=0$ (resp. $\lim_{x \rar r}h(x)=0$). 

Moreover, since $h(X)M$ is a local martingale,  $dM_t= -M_t\frac{\bbA h(X_t)}{h(X_t)}dt$. Thus, $M$  is  given by
\[
M_t=\exp\left(-\int_0^t \frac{\bbA h(X_s)}{h(X_s)}ds\right).
\]

In the light of the above discussion we  now introduce the concept of a {\em recurrent transformation of a diffusion}.
\begin{definition} Let $X$ be a regular  diffusion satisfying  Assumption \ref{a:reg} and  $h:I\to [0,\infty)$ be an absolutely continuous function. Then, $(h,M)$ is said to be a recurrent transform (of $X$) if the following are satisfied:
\begin{enumerate}
\item $M$ is an adapted process of finite variation.
\item $h(X)M$ is a nonnegative local martingale. 
\item The function $s_h$ from (\ref{e:recscale}) is finite for all $x \in (l,r)$ with $-s_h(l+)=s_h(r-)=\infty$.
\item There exists a unique weak solution to (\ref{e:sderectr}) for $t \geq 0$ for any $x \in (l,r)$.
\end{enumerate}
\end{definition}
In the above definition, the defining condition for a recurrent transformation is the function $s_h$ and its explosive nature near the boundaries. The function $h$ and  the  functional $M$ come  into play when one wants to construct a weak solution of the SDE (\ref{e:sderectr}) and show that the law of its solution is locally absolutely continuous with respect to that of the original process $X$, which satisfies (\ref{e:sdeX}). The next theorem, whose proof is delegated to the Appendix, suggests a general machinery for constructing recurrent transformations.
 \begin{theorem} \label{t:rectr}  Let $X$ be a regular  diffusion satisfying Assumption \ref{a:reg}.  Consider an absolutely continuous function $h:I \to [0,\infty)$ such that its left derivative $h'$ is of finite variation. Suppose further that the mapping $s_h$ given by (\ref{e:recscale}) is finite for all $x \in (l,r)$ and that $-s_h(l+)=s_h(r-)=\infty$. Then, the following statements are valid.
 \begin{enumerate}
 \item $h'$ can be chosen to be left-continuous. Moreover, the signed measure defined by $h'$ on $(l,r)$ admits the Lebesgue decomposition $dh'(x)=h''(x)dx + n(dx)$, where $h''$ denote its Borel measurable  Radon-Nikodym derivative with respect to the Lebesgue measure on $(l,r)$, and $n$ is a locally finite signed measure on $(l,r)$ that is singular with respect to the Lebesgue measure.
 \item The integral 
\be \label{e:integrability}
 \chf_{[t<\zeta]}\left(\int_0^{t} \left|\tbA h(X_s)\right|ds + \int_l^r \frac{L^x_{t }}{2}	|n(dx)|\right)<\infty, \; P^y\mbox{-a.s.}, 
 \ee
 for every $y \in (l,r)$, where $\tbA h(x)= \frac{\sigma^2(x)}{2}h''(x)+ b(x)h'(x)$.
 \item $(h,M)$ is a recurrent transform, where, on $[t<\zeta]$,
 \bean
 M_t&:=&\exp\left(-\int_0^{t} \frac{\tbA h(X_s)}{h(X_s)}ds -\int_0^t\frac{1}{h(X_s)}d\Lambda_s(h)\right) \mbox{ and }\\
 \Lambda_t(h)&:=&\int_{(l,r)} \frac{L^x_{t }}{2}n(dx).
 \eean
 \item $\inf\{t>0:h(X_t)M_t=0\}=\zeta, \,P^x$-a.s..
 \item Let $R^{h,x}$ be the law of the solution of (\ref{e:sderectr}) and $F \in \cF_T$ for some $(\cF_t)$-stopping time $T$. Then,
 \be \label{e:ACrectr}
 R^{h,x}(F,T<\infty)=\frac{1}{h(x)}E^{x}\left[\chf_{F}h(X_T)M_T\right].
 \ee
In particular, $h(X)M$ is a $P^x$-martingale.
\item  If $T$ is an $(\cF_t)$-stopping time such that $R^{h,x}(T<\infty)=1$, then for any $ F \in \cF_T$ the following identity holds:
\be \label{e:survivalg}
P^x(\zeta>T, F)= h(x)E^{h,x}\left[\chf_F \frac{1}{h(X_T)M_T}\right],
\ee
where $E^{h,x}$ is the expectation operator with respect to the probability measure $R^{h,x}$.
\end{enumerate}
\end{theorem}

\begin{example} \label{ex:Bes}Suppose $\delta >2$ and consider a $\delta$-dimensional Bessel process on $(0,\infty)$, i.e. a one-dimensional diffusion with the dynamics 
\[
dX_t= 2\sqrt{X_t}dB_t + \delta dt.
\]
The scale function is given by $s(x)=1-x^{\frac{2-\delta}{2}}$. Thus, $X$ is transient and  approaches to $\infty$ as $t \rar \infty$, while $0$ is an inaccessible boundary.

Let  $h(x):=x^{\frac{2-\delta}{4}}$ and define 
\[
M_t:=\exp\left(\frac{(\delta-2)^2}{8}\int_0^t \frac{1}{X_s}ds\right), \qquad t \geq 0.
\]
Then, it follows from Theorem \ref{t:rectr} that $M$ is of finite variation. Moreover,
\[
s_h(x)=\frac{\delta -2}{2}\int_1^x\frac{1}{u} du=\log x, \qquad x >0.
\]
Thus, $-s_h(0)=s_h(\infty)=\infty$, and we conclude that $(h,M)$ is a recurrent transform by invoking Theorem \ref{t:rectr} again. The transformation yields the following SDE for the resulting process
\[
dX_t=   2\sqrt{X_t}dB_t + 2 dt,
\]
which is the SDE for a $2$-dimensional squared Bessel process. Recall (or see p.442 of \cite{RY}) that $0$ is polar for a $2$-dimensional squared Bessel process.
\end{example}
The following proposition gives an important example of a recurrent transformation for transient diffusions, which will  be useful in the sequel. 
\begin{proposition} \label{p:step1} Suppose $X$ is a regular transient diffusion satisfying Assumption \ref{a:reg}.  Let $y\in (l,r)$ be fixed and consider the pair $(h,M)$ defined by
\[
h(x):=u(x,y), \; x \in (l,r),\mbox{ and } M_t=\exp\left(\frac{s'(y)L^y_t}{2u(y,y)}\right).
\]
Then, the following hold:
\begin{enumerate}
\item $(h,M)$ is a recurrent transform for $X$.
\item There exists a unique weak solution to 
\be \label{e:step1}
X_t=x+\int_0^t\sigma(X_s)dB_s + \int_0^t \left\{b(X_s)+\sigma^2(X_s)\frac{u_x(X_s, y)}{u(X_s, y)}\right\} ds, \qquad t \geq 0,
\ee
for any $x \in (l,r)$, where  $u_x$ denotes the first partial left derivative of $u(x,y)$ with respect to $x$. 
\item Moreover, if $R^{h,x}$ denotes the law of the solution and $T$ is a stopping time such that $R^{h,x}(T<\infty)=1$, then for any $ F \in \cF_T$ the following identity holds:
\be \label{e:survival}
P^x(\zeta>T, F)= u(x,y)E^{h,x}\left[\chf_F \frac{1}{u(X_T,y)}\exp\left(-\frac{s'(y)}{2u(y,y)}L^y_T\right)\right],
\ee
where $E^{h,x}$ is the expectation operator with respect to the probability measure $R^{h,x}$.
\end{enumerate}
\end{proposition}
The above is a direct corollary of Theorem \ref{t:rectr} since $n(dx)=-s'(y)\eps_{y}(dx)$ in the Lebesgue decomposition of $du_x(x,t)$ as in Part (1) of Theorem \ref{t:rectr} and $u(\cdot,y)$ is twice differentiable with  $\half\sigma^2(x)u_{xx}(x,y)+b(x)u_x(x,y)=0$ for all $x\neq  y$. 

Proposition \ref{p:step1} is in fact a special case of a more general result that will allow us to construct a large family of recurrent transformations. In order to motivate this more general result note that $u(\cdot,y)$ is the {\em potential\footnote{If $\mu$ is a measure on $(l,r)$, the potential of $\mu$ is the function $x \mapsto \int_{(l,r)} u(x,y)\mu(dy)$ and is denoted by $U\mu$. See Section VI.2 of \cite{BG} for details.}} of the Dirac measure at point $y$. Moreover, it is uniformly integrable being bounded.  Conversely, since $X$ in Assumption \ref{a:reg} is a symmetric diffusion, it is well-known (see, e.g., Theorem VI.2.11 in  \cite{BG}) any uniformly integrable potential $h$ is the potential of some measure $\mu$ on $(l,r)$, i.e. $h(x)=\int_{(l,r)} u(x,y)\mu(dy)$. Also note that if $h\not\equiv 0$ is a uniformly integrable potential, e.g. $h=u(\cdot,y)$, then $h(X)$ is a supermartingale, which is not a martingale. As a matter of fact, in view of the Riesz representation of excessive functions (see Theorem VI.2.11 in conjunction  with Proposition IV.5.4 in \cite{BG}) the greatest uniformly integrable harmonic function dominated by $h$ is $0$. The next result, whose proof is in the Appendix, shows that the potential of a probability measure on $(l,r)$ gives rise to a recurrent transform under an integrability condition. 
\begin{theorem} \label{t:rtrpot} Let $\mu$ be a Borel probability measure on $(l,r)$  such that $\int_{(l,r)}|s(y)|\mu(dy)<\infty$. Suppose $X$ is a regular  transient diffusion satisfying Assumption \ref{a:reg} and define 
	\[
	h(x):=\int_{(l,r)} u(x,y)\mu(dy).
	\]
\begin{enumerate}
	\item The left derivative $h'$ of $h$ exists and $(h,M)$ is a recurrent transform of $X$, where
	\[
	M_t :=\exp\left(\int_0^t\frac{1}{h(X_s)}dA_s\right) \mbox{ and } A_t:= \int_{(l,r)} \frac{s'(x)L^x_t}{2}\mu(dx).
	\]
	\item If $R^{h,x}$ denotes the law of the solution of (\ref{e:sderectr}) and $T$ is a stopping time such that $R^{h,x}(T<\infty)=1$, then for any $ F \in \cF_T$ the following identity holds:
	\[
	P^x(\zeta>T, F)= h(x)E^{h,x}\left[\chf_F \frac{1}{h(X_T)}\exp\left(-\int_0^t\frac{1}{h(X_s)}dA_s\right)\right],
	\]
	where $E^{h,x}$ is the expectation operator with respect to the probability measure $R^{h,x}$.
\end{enumerate}
\end{theorem}
\begin{remark}
	Note that $u(\cdot,y)$ satisfies the assumptions of the above theorem since $\mu=\eps_y$ and $\int_{(l,r)} u(x,z)\mu(dz)=u(x,y)<\infty$ for all $x\in (l,r)$. Thus, Proposition \ref{p:step1} is a direct consequence of Theorem \ref{t:rtrpot} as well.
\end{remark}
The next example of a recurrent transform that we shall consider in this paper is obtained via  the $\alpha$-potential density, $u^{\alpha}$ of $X$. In contrast with the previous transform, which only exists for transient diffusions, the next transform can be applied to all regular diffusions. Moreover, the resulting diffusion  will be {\em positive recurrent}. 
\begin{proposition} \label{p:AOtr} Suppose $X$ is a regular  diffusion satisfying Assumption \ref{a:reg}.  Let $y\in (l,r)$ and $\alpha >0$ be fixed and consider the pair $(h,M)$ defined by
\[
h(x):=u^{\alpha}(x,y), \; x \in (l,r),\mbox{ and } M_t=\exp\left(-\alpha t+ \frac{s'(y)L^y_t}{2u^{\alpha}(y,y)}\right).
\]
Then, the following hold:
\begin{enumerate}
\item $(h,M)$ is a recurrent transform for $X$.
\item There exists a unique weak solution to 
\be \label{e:AOtr}
X_t=x+\int_0^t\sigma(X_s)dB_s + \int_0^t \left\{b(X_s)+\sigma^2(X_s)\frac{u^{\alpha}_x(X_s, y)}{u^{\alpha}(X_s, y)}\right\} ds, \qquad t \geq 0,
\ee
for any $x \in (l,r)$, where  $u^{\alpha}_x$ denotes the first partial left derivative of $u^{\alpha}(x,y)$ with respect to $x$. 
\item Moreover, the diffusion defined by the solutions of (\ref{e:AOtr}) is positive recurrent and its stationary distribution on $(l,r)$ is given by
\be \label{e:invdist}
\pi(dx)=\frac{(u^{\alpha}(x,y))^2}{\int_0^{\infty}se^{-\alpha s}p(s,y,y)ds}m(dx),
\ee
where $(p(t,\cdot,\cdot))_{t >0}$ is the transition density of the original diffusion with respect to its speed measure $m$. 
\end{enumerate}
\end{proposition}
As in the case of Proposition \ref{p:step1}, parts (1) and (2) of the above result is a direct corollary of Theorem \ref{t:rectr} but will also be a special case of a more general theorem in terms of $\alpha$-potentials. Analogously,  $u^{\alpha}(\cdot,y) $ of $X$ is the $\alpha$-potential of the Dirac measure at $y$ and $(e^{-\alpha t}u^{\alpha}(X_t,y))$ is a uniformly integrable supermartingale converging a.s. to $0$ as $t \rar \zeta$.  Moreover,  any uniformly integrable $\alpha$-potential is of the form $\int_l^r u^{\alpha}(x,y)\mu(dy)$ for some measure on $(l,r)$.
\begin{theorem} \label{t:prctr}
	 Suppose $X$ is a regular  diffusion satisfying Assumption \ref{a:reg} and $\alpha>0$. Let $\mu$ be a Borel probability measure on $(l,r)$ such that $\int_{(l,r)} u^{\alpha}(y,y)d\mu(y)<\infty$. Define 
	 \[
	 h(x):=\int_{(l,r)}u^{\alpha}(x,y)\mu(dy).
	 \]
	 \begin{enumerate}
	 	\item The left derivative $h'$ of $h$ exists and $(h,M)$ is a recurrent transform of $X$, where
	 	\[
	 	M_t :=\exp\left(-\alpha t+\int_0^t\frac{1}{h(X_s)}dA_s\right) \mbox{ and } A_t:= \int_{(l,r)} \frac{s'(x)L^x_t}{2}\mu(dx).
	 	\]
	 	\item Moreover, if there exists $\eps >0$ such that $\int_{(l,r)} u^{\alpha-\eps}(y,y)\mu(dy)<\infty$, then the diffusion defined by the  solutions of (\ref{e:sderectr}) is positive recurrent and its stationary distribution on $(l,r)$ is given by 
	 	\[
	 	\pi(dx)=\frac{h^2(x)}{\int_l^r h^2(y)m(dy)}m(dx).
	 	\]
	 \end{enumerate} 
\end{theorem}

\begin{remark}
	Note that $u^{\alpha}(\cdot,y)$ satisfies the assumptions of the above theorem since $\mu=\eps_y$ and $\int_{(l,r)} u^{\alpha-\eps}(x,z)\mu(dz)=u^{\alpha-\eps}(x,y)<\infty$ for all $\eps\in [0,\alpha)$. Thus Proposition \ref{p:AOtr} follows directly from Theorem \ref{t:prctr}.
	
	 Moreover, if $X$ is transient, the potential density $u$ exists and is finite.  In this case the condition $\int_{(l,r)} u(y,y)\mu(dy)<\infty$ is equivalent to $\int_{(l,r)} |s(y)|\mu(dy)<\infty$ under the assumption that $\mu$ is a probability measure. Thus,   the condition $\int_{(l,r)} u^{\alpha}(y,y)\mu(dy)<\infty$ in Theorem \ref{t:prctr} is the exact analogue of the condition $\int_{(l,r)} |s(y)|\mu(dy)<\infty$ of Theorem \ref{t:rtrpot}.
	 
	 If $f$ is nonnegative, $\int_l^r f(x)m(dx)=1$, and $\int_l^r f(x)u^{\alpha}(x,x)m(dx)<\infty$,  Theorems \ref{t:rtrpot} and \ref{t:prctr} show that $h(x):=U^{\alpha}f(x)$ will define a recurrent transform for $\alpha\geq 0$. In this case the finite variation process $A$ will be given by
	\[
	A_t=\int_0^t f(X_s)ds.
	\]	
	
	For instance, in  Example \ref{ex:Bes} it can be verified using the scale function and the speed measure of squared Bessel processes that $h(x)=\frac{(\delta-2)^2}{8}\int_0^{\infty}u(x,y)y^{-\frac{\delta +2}{4}}m(dy)$ leading to $dA_t=X_t^{-\frac{\delta +2}{4}}dt$ in the notation of Theorem \ref{t:rtrpot}.
\end{remark}
\begin{example} \label{ex:BMalt} Suppose $X$ is a standard Brownian motion. It is well-known that
\[
u^{\alpha}(x,y)=\frac{1}{\sqrt{2 \alpha}}\exp\left(-\sqrt{2 \alpha}|x-y|\right).
\]
Thus, if we use  the transform in Proposition \ref{p:AOtr} with $y=0$, the recurrent transform is the solution to the following SDE:
\[
dX_t= dB_t -\sqrt{2 \alpha}\sgn(X_t)dt,
\]
where $\sgn(x)=-\chf_{[x<0]}+\chf_{[x\geq 0]}$. This is a Brownian motion with alternating state-dependent drift, which plays a key role in the so-called {\em bang-bang} control problem (see Section 6.6.5 in \cite{KS} and the references therein).
\end{example}

We shall  consider in  subsequent sections the applications of the above recurrent transforms to optimal stopping as well as some pricing issues arising in Black-Scholes models when the stock  price follows a strict local martingale. However, one can find an immediate application of the recurrent transform to the computation of the distribution of the first exit time for a one-dimensional diffusion from an interval. Indeed,  such a first exit time can always be viewed as the life time of a transient diffusion by killing the original one as soon as it exits the given interval. Thus, the problem reduces to finding $P^x(\zeta>t)$ for all $t>0$, where $P^x$ is the law of the transient diffusion starting at $x$ and $\zeta$ is its lifetime, i.e. the first time it exits the given interval.  The following is a direct consequence of Proposition \ref{p:step1}.
\begin{corollary} \label{c:survfunction}  Let $X$ be a regular transient diffusion satisfying Assumption \ref{a:reg}.  Then
\[
P^x(\zeta>t)=  u(x,y)E^{h,x}\left[\frac{1}{u(X_t,y)}\exp\left(-\frac{s'(y)}{2u(y,y)}L^y_t\right)\right],
\]
where $E^{h,x}$ is the expectation operator with respect to the law of the recurrent transform given by (\ref{e:step1}). 
\end{corollary}
Although the above formula does not in general give $P^x(\zeta>t)$ in closed-form, it is nevertheless practical. Indeed, by running a Monte-Carlo simulation of the solution of (\ref{e:step1}), one can get a close estimate of
\[
E^{h,x}\left[\frac{1}{u(X_t,y)}\exp\left(-\frac{s'(y)}{2u(y,y)}L^y_t\right)\right]
\]
by approximating the local time using the occupation times formula. 

Karatzas and Ruf \cite{KR} have shown that the function $v(t,x):=P^x(\zeta>t)$ is the smallest nonnegative classical supersolution of 
\be \label{e:cauchy0}
v_t=Av, \qquad v(0,\cdot)=1
\ee
under the assumption that $\sigma$ and $b$ are locally uniformly H\"older continuous on $(l,r)$.  Thus,  combining their Proposition 5.4 and Corollary \ref{c:survfunction} we deduce the following.
\begin{corollary}  Let $X$ be a regular transient diffusion satisfying Assumption \ref{a:reg}. Assume further that $\sigma$ and $b$ that appears in (\ref{e:sdeX}) are locally uniformly H\"older continuous on $(l,r)$. Define
\[
v(t,x):= u(x,y)E^{h,x}\left[\frac{1}{u(X_t,y)}\exp\left(-\frac{s'(y)}{2u(y,y)}L^y_t\right)\right],
\]
where $E^{h,x}$ is the expectation operator with respect to the law of the recurrent transform given by (\ref{e:step1}).  Then, $v$ is the smallest nonnegative classical supersolution of (\ref{e:cauchy0}).
\end{corollary}
\begin{remark}
	In fact there is not a unique way of representing the minimal nonnegative classical supersolutions of (\ref{e:cauchy0}). Indeed, if $h$ is the potential of a probability measure $\mu$ on $(l,r)$ satisfying the hypothesis of Theorem \ref{t:rtrpot}, then
	\[
	P^x(\zeta>t)=  h(x)E^{h,x}\Big[\frac{1}{h(X_t)}\exp\Big(-\int_{(l,r)}\frac{s'(y)L^y_t}{2h(y)}\mu(dy)\Big)\Big].
	\] 
	In particular if $\mu(dy)=f(y)m(dy)$ for some $f$, $\int_{(l,r)}\frac{s'(y)L^y_t}{2h(y)}\mu(dy)=\int_0^t \frac{f(X_s)}{h(X_s)}ds$. 
\end{remark}
\begin{remark} \label{r:simulation}
	The recurrent transformation of a transient diffusion can be used to improve the accuracy of discrete Euler approximations of diffusions that are killed when leaving a bounded interval $[a,b]$. Suppose $\zeta$ represents the first exit time from this interval and one is interested in the Monte Carlo simulation of $E^x[F(X_T)\chf_{[T<\zeta]}]$ for some suitable $F$ via a discrete Euler scheme applied to the SDE (\ref{e:sdeX}) for $X$. Gobet \cite{GobetKilled} has shown that the discretisation error is of order $N^{-\half}$, where $N$ is the number of discretisations. This order of convergence is exact and intrinsic to the killing. However, this corresponds to a loss of accuracy compared to the standard Euler scheme applied to a diffusion without killing, where the error is of order $N^{-1}$. On the other hand, the recurrent transformation from Theorem \ref{t:rtrpot} can be used to improve the convergence rate back to $N^{-1}$ since
	\[
	E^x[F(X_T)\chf_{[T<\zeta]}]= h(x)E^{h,x}\Big[F(X_T) \frac{1}{h(X_T)}\exp\Big(-\int_0^T\frac{f(X_s)}{h(X_s)}ds\Big)\Big],
	\]  
	where $h(x)=\int_l^r u(x,y)f(x)m(dx)$ for a nonnegative $f$ with $\int_l^r f(x)m(dx)=1$. This is due to the fact that there is no killing under $R^{h,x}$, i.e. $R^{h,x}(\zeta=\infty)=1$. We will study in more detail the improvement of the discrete Euler scheme for killed diffusions in a subsequent paper.
\end{remark}
\subsection{Connection with Doob's $h$-transform} \label{s:rvsh} It is trivial to check that $(h,M)$-recurrent transform of $X$ has $h^2dm$ as its speed measure. In the specific case considered in Proposition \ref{p:step1} the recurrent transform is a one-dimensional diffusion with scale
\[
s_h(x)=\int_c^x \frac{s'(z)}{(u(z,y))^2}dz,
\]
and the speed measure $(u(z,y))^2m(dz)$. Note that this is not the only diffusion with this scale function and the speed measure. Indeed, if one considers the $h$-transform of $X$ via $h(x)=\frac{u(x,y)}{u(y,y)}$, one obtains a diffusion which amounts to conditioning the paths of $X$ to converge to $y$ and killed at its last exit from $y$. The resulting diffusion is obviously a transient diffusion but has the same scale and the speed (see, e.g. Theorem 6.2 in \cite{EH} or Paragraph 31 in Section II.5 of \cite{BorSal}). The crucial difference between the two transformations is that the $h$-transform  involves killing while the recurrent transform does not. 

Killing of the trajectories in the $h$-transform is also apparent from the following representation. Denoting the law of the $h$-transform by $\tilde{P}^{u,x}$ we deduce 
\[
\tilde{E}^{u,x}[F\chf_{[\zeta>t]}]=\frac{E^x[F u(X_t,y)]}{u(x,y)}=\frac{E^x\left[F\chf_{[G_y>t]}\right]}{h(x)}.
\]
In the above $F$ is an $\cF_t$-measurable random variable and $G_y:=\sup\{t:X_t=y\}$ (see Section 3.9 -- in particular the expression (3.211)-- in \cite{MR} for the  details). The above identity in particular implies
\[
\tilde{P}^{u,x}(\zeta>t)=\frac{P^x(G^y>t)}{h(x)}, \; \forall t\geq 0,
\]
i.e., $\tilde{P}^{u,x}$-distribution of $\zeta$ coincides with the law of $G^y$ under $P^x$ after a normalisation. Observe that $P^x(G_y<\zeta)=1$ since $X$ is transient under $P^x$. 

Given this close relationship between the recurrent transform and the $h$-transform one may wonder whether it is possible to obtain the latter from the former via a killing. This is in fact possible. Indeed, for any $\cF_t$-measurable bounded random variable $F$, one has
\be \label{e:killRT}
E^{h,x}\left[F \exp\left(-\frac{s'(y)}{2 u(y,y)}L^y_t\right)\right]=E^{x}\left[F \frac{u(X_t,y)}{u(x,y)}\right]=\tilde{E}^{u,x}\left[F\chf_{[\zeta>t]}\right].
\ee
Thus, if one kills the trajectories of the recurrent transform at rate $\frac{s'(y)}{2 u(y,y)}L^y_t$, then one obtains the $h$-transform. As such, $h$-transform is subordinate (see Section III.2 of \cite{BG} for a description of subordinate semigroups) to the recurrent transform, i.e. $\tilde{E}^{u,x}\left[F\right]\leq  E^{h,x}\left[F\right]$ for all nonnegative  $\cF_t$-measurable  $F$  that vanishes on $[\zeta,\infty)$. 

We shall next describe how one can implement this killing in practice. To this end define $S_a:=\inf\{t\geq 0:L_t^y >\frac{2 u(y,y)a}{s'(y)}\}$ for $a > 0$ and consider a unit exponential random variable $\alpha$ that is independent from the recurrent process. Then, for any $\cF_t$-measurable bounded random variable $F$
\bean
E^{h,x}\left[F\chf_{[t<S_{\alpha}]}\right]&=&\int_0^{\infty}e^{-a}E^{h,x}\big[F\chf_{[L^y_t\leq \frac{2 u(y,y)a}{s'(y)}]}\big]da=E^{h,x}\left[\int_{\frac{s'(y)L^y_t}{2  u(y,y)}}^{\infty}Fe^{-a} da\right]\\
&=&E^{h,x}\left[F \exp\left(-\frac{s'(y)}{2 u(y,y)}L^y_t\right)\right],
\eean
yielding the relationship in (\ref{e:killRT}).

Note that if the $h$-transform is given by a bounded potential $h$ as in Theorem \ref{t:rtrpot}, similar considerations also show that the $h$-transform can be obtained from the recurrent transform by killing the recurrent transform at the first time that $\int_0^{\cdot}\frac{1}{h(X_s)}dA_s$, where $A$ is the finite variation process  associated to the recurrent transform via Theorem \ref{t:rtrpot}, exceeds a unit exponential time. We leave the easy details to the reader. This in turn gives a very useful recipe for the simulation of $h$-transforms, whose lifetimes often correspond to  some last passage times that are not stopping times (see Remark 11.27 in \cite{ChungWalsh}). 

\section{Limits of recurrent transforms and strict local martingales}\label{s:limit}
Motivation of this section comes from the financial models that we shall treat in more detail in Section \ref{s:BS}. Consistent with the setting therein $X$ will assumed to be a non-negative diffusion in natural scale in this section.  As our focus is on strict local martingales this necessitates the choice of $r=\infty$. We also translate $X$ so that $l=0$. Consequently, $u(x,y)=x\wedge y$ and the recurrent transform in (\ref{e:step1}) reads
\be \label{e:LMrt}
X_t=x+\int_0^t \sigma(X_s)dB_s+\int_0^t\frac{\sigma^2(X_s)}{X_s}\chf_{[X_s\leq y]}ds, \quad x>0.
\ee
We have established in Proposition \ref{p:step1} that the above SDE has a non-explosive weak solution that is unique in law. Moreover, the solution never hits $0$.  If $(X^y)_{y >0}$ denotes the solutions of (\ref{e:LMrt}) indexed by $y$, we notice immediately that the drift term associated to $X^y$ is increasing in $y$. Thus, if the solutions are strong, we may hope that the solutions are increasing in $y$ under a mild hypothesis on $\sigma$. Then, if we let $Y_t:=\lim_{y \rar \infty}X^y_t$, the resulting limit is expected to satisfy
\be \label{e:limitRT}
Y_t=x+\int_0^t \sigma(Y_s)dB_s+\int_0^t\frac{\sigma^2(Y_s)}{Y_s}ds, \quad x>0.
\ee
Since $Y$ is obtained as an increasing limit of $X^y$, it will never hit $0$. However, its behaviour near the infinite boundary, and in particular whether it may explode in finite time, requires a further look. We shall in fact see that $Y$ is the SDE satisfied by the $h$-transform of $X$, where $h(x)=x$, and its explosive behaviour depends exclusively on the strict local martingale property of $X$. 

The next assumption will be sufficient to ensure that the solutions of (\ref{e:LMrt}) are strong and increase in $y$. Note that  one could get the existence and uniqueness of strong solutions under weaker hypothesis. However, the following stronger condition is imposed  since we are also interested in a comparison result for the strong solutions.
\begin{assumption}\label{a:sholder} There exists a strictly increasing function $\rho:[0,\infty)\to [0,\infty)$ with 
\[
\int_{0+}^{\infty}\frac{1}{\rho(a)}da=\infty
\]
such that
\[
(\sigma(x)-\sigma(y))^2 \leq \rho(|x-y|), \qquad x\neq y.
\]
\end{assumption} 
As we will be working with strong solutions in this section let us fix a Brownian motion, $\beta$, on a fixed probability space $(\Om, \cF, (\cF_t), \bbP)$, where $(\cF_t)_{t \geq 0}$ is as in Section \ref{s:prelim}, so that  
\be \label{e:XLM}
X_t= x+\int_0^t \sigma(X_s)d\beta_s, \quad x>0.
\ee
It follows from  Theorem IX.3.5 in \cite{RY} and Corollary 5.3.23 in \cite{KS} that $X$ is the unique strong solution of  (\ref{e:XLM}) under Assumptions \ref{a:reg} and \ref{a:sholder}.

What we would like to achieve next is to pass to a locally absolutely continuous measure, which will support all the solutions of (\ref{e:LMrt}). The next result does not need Assumption \ref{a:sholder}.
\begin{proposition} \label{p:labcQ} Suppose that Assumption \ref{a:reg} is in force and $X$ satisfies (\ref{e:XLM}) on $(\Om, \cF, (\cF_t), \bbP)$ supporting the Brownian motion, $\beta$. There exists a $\bbQ$ on $(\Om, \cF)$  and a sequence of stopping times $(\tau_n)_{n \geq 1}$ such that i) $\lim_{n \rar \infty}\bbQ(\tau_n\leq t)=0$, ii) $\bbQ |_{\cF_{\tau_n}} \ll \bbP |_{\cF_{\tau_n}}$ and iii)
\[
X_t= x+\int_0^t \sigma(X_s)dB_s +\int_0^t\frac{\sigma^2(X_s)}{X_s}\chf_{[X_s\leq 1]}ds,
\]
where $B$ is a $(\Om, \cF, (\cF_t), \bbQ)$-Brownian motion.
\end{proposition}
\begin{proof}
Consider the $(h,M)$ transform in Proposition \ref{p:step1}, where $y=1$, and set $\tau_n:=\inf\{t\geq 0:L^y_t \geq n\}$.  Then,
$h(X_{t \wedge \tau_n}) M_{t \wedge \tau_n}$ is a bounded martingale that defines a $\bbQ_n$ on $\cF_{\tau_n}$. Note that $\bbQ_n(\tau_n \leq t)= R^{h,x}(L^y_t \geq n)$ using the notation of Proposition \ref{p:step1}. Thus, $\lim_{n \rar \infty}Q_n(\tau_n \leq t)=\lim_{n \rar \infty}R^{h,x}(L^y_t \geq n)=R^{h,x}(L^y_t= \infty)=0$, and i) and ii) follow from Theorem 1.3.5 in \cite{SV}. 

Moreover, since $\bbQ$ agrees with $\bbQ_n$ on $\cF_{\tau_n}$, we have
\[
X_t= x+\int_0^t \sigma(X_s)dB_s +\int_0^t\frac{\sigma^2(X_s)}{X_s}\chf_{[X_s\leq 1]}ds, \quad t<\tau_n,
\]
where 
\[
B_t= \beta_t-\int_0^t \frac{\sigma(X_s)}{X_s}\chf_{[X_s\leq 1]}ds, \quad t <\tau_n.
\]
As such, $B$ is a Brownian motion stopped at $\tau_n$. Invoking the fact that $\lim_{n \rar \infty} \bbQ(\tau_n \leq t)=0$ yields iii).
\end{proof}

The above proposition constructs a locally absolutely continuous probability measure, $\bbQ$, and a $\bbQ$-Brownian motion, $B$.  Thus, once we impose Assumption \ref{a:sholder}, (\ref{e:LMrt}) will possess the pathwise uniqueness property by virtue of Proposition IX.3.1 and Lemma IX.3.1 in \cite{RY}. Combining this with  Corollary 5.3.23 in \cite{KS} we arrive at the following.
\begin{proposition} Suppose that Assumptions \ref{a:reg} and \ref{a:sholder} hold.  Let $B$ and $(\Om, \cF, (\cF_t), \bbQ)$ be as in Proposition \ref{p:labcQ}. Then, for each $y>0$, there exists a unique strong solution to (\ref{e:LMrt}).
\end{proposition}
As mentioned earlier Assumption \ref{a:sholder} will also imply that the solutions of (\ref{e:LMrt}) are increasing in $y$.
\begin{proposition} Suppose that Assumptions \ref{a:reg} and \ref{a:sholder} hold.  Let $B$ and $(\Om, \cF, (\cF_t), \bbQ)$ be as in Proposition \ref{p:labcQ} and denote by $X^y$ the unique strong solution of (\ref{e:LMrt}). Then, $\bbQ(X^{y_0}_t\leq X^{y_1}_t, \, \forall t \geq 0)=1$ whenever $y_0\leq y_1$.  
\end{proposition}
\begin{proof}
Let $b_i(x)=\frac{\sigma^2(x)}{x}\chf_{[x\leq y_i]}$ for $i=0,1$, and define $b^{\eps}(x)=b_1(x)+\eps.$ Observe that for any sufficiently small $\delta >0$  there exists a Lipschitz function, $g$, such that $b_0(x) \leq g(x)\leq b^{\eps}(x)$ for $x > \delta$ due to the continuity of $\sigma$. Thus, it follows from Theorem 1.1 in Chap. VI of \cite{IW} that $X^{y_0}_t \leq Z^{\eps}_t$ for all $t <T_{\delta}$, where $T_{\delta}=\inf\{t\geq 0: X^{y_0}_t  \leq \delta\}$ and
\[
Z^{\eps}_t =x + \int_0^t \sigma(Z^{\eps}_s)dB_s + \int_0^t b^{\eps}(Z^{\eps}_s)ds.
\]
Note that since $\sigma$ satisfies (\ref{e:ESR}) and Assumption \ref{a:sholder} the above SDE has a unique strong solution. Since $\delta$ is arbitrary and $\lim_{\delta\rar 0}T_{\delta}=\infty, \bbQ$-a.s., we immediately deduce that $X^{y_0}_t \leq Z^{\eps}_t$ for all $t\geq 0$.  Next, we claim that $Z^{\eps}_t \rar X^{y_1}_t$ as $\eps \rar 0$ for $t <T_{\delta}$. 

Indeed, we can again find a Lipschitz continuous function between $b^{\eps_0}$ and $b^{\eps_1}$ whenever $\eps_0<\eps_1$ on $(\delta, \infty)$ for any $\delta >0$. Therefore, the same theorem in \cite{IW} yields that $Z^{\eps}_t$ is increasing in $\eps$ for each $t>0$. Set $Z_t=\lim_{\eps \rar 0}Z^{\eps}_t$. It follows from the continuity of $\sigma$ and the  dominated convergence theorem for stochastic integrals that
\[
\lim_{\eps \rar 0}\int_0^t \sigma(Z^{\eps}_s)dB_s=\int_0^t \sigma(Z_s)dB_s.
\]
Also observe that $Z_t <y_1$ if and only if $Z^{\eps}_t<y_1$ for all but finitely many $\eps$ (number possibly depending on $\om$) since $Z^\eps$ is decreasing to $Z$ as $\eps \rar 0$. Thus, $b^{\eps}(Z^{\eps}_t)\rar b_1(Z_t)$ as $\eps \rar 0$ for each $t>0$. Since $b^{\eps_n}$ is uniformly bounded on $(\delta, \infty)$ given any $(\eps_n)_{n \geq 1}$ converging to $0$, we deduce from Lebesgue's dominated convergence theorem that
\[
\lim_{\eps \rar 0}\int_0^{t\wedge S_{\delta}} b^{\eps}(Z^{\eps}_s)ds= \int_0^{t\wedge S_{\delta}}  b_1(Z_s)ds,
\]
where $ S_{\delta}=\inf\{t>0:Z_t <\delta\}$. Thus, we have shown that $Z$ solves (\ref{e:LMrt}) with $y=y_1$ up to  $S_{\delta}$. Since $X^{y_1}$ is the unique solution of this equation, we therefore establish that $X^{y_1}_t=\lim_{\eps \rar 0}Z^{\eps}_t$ for $t \leq S_{\delta}= \inf\{t>0:X^{y_1}_t <\delta\} $.  Therefore,  $X^{y_0}_t \leq X^{y_1}_t$ for $t <T_{\delta}$. As before, we can pass to the limit as $\delta \rar 0$ and concludefor every $t \geq 0$ that   $X^{y_0}_t \leq X^{y_1}_t$ . Moreover, due to the continuity of $X^{y_i}$s, we may choose a null set independent of $t$ to deduce  $\bbQ(X^{y_0}_t\leq X^{y_1}_t, \, \forall t \geq 0)=1$.
\end{proof}

Thanks to the above result $X^y$ is increasing in $y$ and we can define $Y_t= \lim_{y \rar \infty}X^y_t$. Moreover, the arguments used in the proof of the above proposition yields the following corollary. 
\begin{corollary}\label{c:rtlimit} Suppose that Assumptions \ref{a:reg} and \ref{a:sholder} hold and let $X^y$ be the unique strong solution of (\ref{e:LMrt}), where $B$ and $(\Om, \cF, (\cF_t), \bbQ)$ are as in Proposition \ref{p:labcQ}. Then, $Y$ is the unique strong solution of (\ref{e:limitRT}), where $Y_t=\lim_{y \rar \infty}X^y_t$.
\end{corollary}
It can be checked easily that the scale function of the diffusion in (\ref{e:limitRT}) is $1-\frac{1}{x}$. Thus, the solution never hits $0$ and diverges to $\infty$ as $t \rar \infty$. Whether the explosion happens in finite time depends on the martingale property of $X$. Note that if one is content with weak solutions, (\ref{e:limitRT}) has a unique weak solution when $\sigma$ satisfies (\ref{e:ESR}).
\begin{proposition} \label{p:Xvshtr} Suppose that $\sigma$ satisfies (\ref{e:ESR}) and consider a weak solution,  $Y$, of  (\ref{e:limitRT}). Let $Q^x$ be the law of the solution of (\ref{e:limitRT}). Then, $Q^x(\lim_{t \rar \infty}Y_t=\infty)=Q^x(Y_t>0, \; \forall t>0)=1$. In particular, $\zeta=\inf\{t: Y_t =\infty\},\, Q^x$-a.s. for each $x >0$. Moreover, $Q^x(\zeta=\infty)=1$ if and only if $X$ is a martingale, where $X$ is given by (\ref{e:XLM}).
\end{proposition}
\begin{proof}
Note that the scale function of $Y$ after our normalisation is given by $s(x) =1-\frac{1}{x}$. Thus, (\ref{e:psiif}) applies and we deduce $Q^x(\lim_{t \rar \infty}Y_t=\infty)=Q^x(Y_t>0, \; \forall t>0)=1$. Since $\zeta$ is the lifetime of the diffusion, this also implies that $\zeta=\inf\{t: Y_t =\infty\},\, Q^x$-a.s.. 

Next, it follows from Theorem 5.5.29 and Problem 5.5.27 in \cite{KS} that $Q^x(\zeta=\infty)=1$ if and only if
\[
\lim_{x \rar \infty}\int_1^{x}\frac{x-z}{x}\frac{z}{\sigma^2(z)}dz =\infty.
\]
However, 
\[
\int_1^{x}\frac{x-z}{x}\frac{z}{\sigma^2(z)}dz=\frac{1}{x}\int_c^x\int_c^y \frac{z}{\sigma^2(z)}dzdy.
\]
Thus, the above limit is valid  if and only if 
\[
\int_1^{\infty}\frac{z}{\sigma^2(z)}dz =\infty,
\]
which is well-known to be equivalent to the martingale property of $X$ (see, e.g., Theorem 1.4 in \cite{DelShir} under a mild assumption on $\sigma$ or Theorem 1 in \cite{Kotani} for a general result).
\end{proof}
\begin{remark} \label{r:Yashtr} Using the methods employed in the proof of Theorem \ref{t:rectr} one can show that the law of (\ref{e:limitRT}) is equal to that of the $h$-transform of $X$, where $h(x)=x$. The relationship between the martingale property of $X$ and the finiteness of the explosion time of its $h$-transform, i.e. Proposition \ref{p:Xvshtr}, has already been observed in the literature (see, e.g., \cite{F72} or, more recently, \cite{KKN}). 
\end{remark}
As observed earlier $1-1/x$ is a scale function of $Y$. Consequently, $1/Y$ is a nonnegative local martingale. It turns out that the martingale property of $1/Y$ is determined by whether $X$ hits $0$ or not.
\begin{proposition}\label{p:Yinvmart}
Suppose that $\sigma$ satisfies (\ref{e:ESR}) and let  $Y$ be a weak solution of  (\ref{e:limitRT}), whose law is denoted by $Q^x$. Then, $\frac{1}{Y}$ is a $Q^x$ martingale if and only if $\bbP(X_t>0,\; \forall t>0)=1$, where $X$ is given by (\ref{e:XLM}). 
\end{proposition}
\begin{proof}
	Denote $\frac{1}{Y}$ by $\xi$. Then, $d\xi_t=\sigma(\frac{1}{\xi_t})\xi^2_tdB_t$	for some Brownian motion $B$. It follows from Theorem 1 in \cite{Kotani} that  $\xi$ is a martingale if and only if
	\[
	\int_1^{\infty}\frac{1}{\sigma^2(\frac{1}{z})z^3}dz=\infty.
	\]
	However, after a change of variable the above condition is equivalent to
	\[
	\int_0^1\frac{x}{\sigma^2(x)}dx=\infty,
	\]
	which is equivalent to the strict positivity of $X$ by Theorem 5.5.29 in \cite{KS}.
	\end{proof}
\section{Yet another transform for recurrent diffusions} \label{s:another}
We have noted in Section \ref{s:rectr} a remarkable transform that turned any regular diffusion into a positively recurrent one. This section will present a particular type of transformation for recurrent diffusions that will render them transient. This transformation will be especially useful when we consider the optimal stopping problems in Section \ref{s:OS}.

When $X$ is a transient diffusion with $s(l)=0$, it converges to $l$ with positive probability. If one wants to condition this process to converge to $r$ with probability $1$, it suffices to use the $h$-transform with $h=s$ (see, e.g. Section 6 in \cite{EH}). If $X$ is recurrent, on the other hand, the range of $s$ is the whole real line so one needs to consider taking absolute values to obtain a positive local martingale using $s$.  The next proposition introduces a particular conditioning for recurrent diffusions that conditions $X_{\infty}$ to exist and take values in the set $\{l,r\}$.  Similar to the martingale characterisation of a  positive diffusion in natural scale in terms of the explosion time of its $h$-transform that we have seen in Proposition \ref{p:Xvshtr}, the resulting diffusion will turn out to have a finite explosion time if and only if $s(X)$ is a strict local martingale. 
\begin{proposition} \label{p:ttforr} Suppose $X$ is a recurrent diffusion satisfying Assumption \ref{a:reg}. Let $c>0$ be fixed and $y^*$ be the unique point in $(l,r)$ such that $s(y^*)=0$. Then, the following statements are valid:
\begin{enumerate}
\item $N$ is a local martingale, where
\[
N_t:= \left(1+c|s(X_t)|\right)\exp\left(-cs'(y^*)L^{y^*}_t\right).
\]
\item For any $x \in (l,r)$ there exists a unique weak solution to 
\be \label{e:ttforr}
X_t= x+ \int_0^t \sigma(X_s)dB_s +\int_0^t \left\{b(X_s) - c\frac{s'(X_s)}{1-cs(X_s)}\chf_{[X_s\leq y^*]}+c\frac{s'(X_s)}{1+cs(X_s)}\chf_{[X_s> y^*]}\right\}ds, \quad t <\zeta,
\ee
where $\zeta:=\inf\{t: X_{t-}\in \{l,r\}\}$.
\item The regular diffusion defined by (\ref{e:ttforr}) has scale function
\be \label{e:ttscale}
\tilde{s}(x):=\frac{1+c(s(x)+|s(x)|)}{2(1+c|s(x)|)},
\ee
and speed measure
\[
\tilde{m}(dx)=\frac{4(1+c|s(x)|)^2}{c\sigma^2(x)s'(x)}dx=\frac{2(1+c|s(x)|)^2}{c}m(dx).
\]
 $\tilde{P}^x(X_{\zeta}=r)=\tilde{s}(x)=1-\tilde{P}^x(X_{\zeta}=l)$, where $\tilde{P}^x$ denotes the law of (\ref{e:ttforr}). Moreover, $\tilde{P}^x(\zeta=\infty)=1$ if and only if $s(X)$ is a $P^x$-martingale. 
 \item For any $F\in \cF_t$ the following absolute continuity relationship holds.
 \be \label{e:ttforrAC}
 \tilde{P}^x(F, \zeta>t)=\frac{E^x\left[\chf_F N_t\right]}{1+c|s(x)|}.
 \ee
 Consequently, $N$ is a martingale if and only if $s(X)$ is.
\end{enumerate}
\end{proposition}
\begin{proof}
Note that $1+c|s(x)|$ is absolutely continuous with a jump in its left derivative at $x=y^*$ with size $2cs'(y^*)$. Thus, $N$ is a local martingale due It\^{o}-Tanaka formula as in Proposition \ref{p:step1}. Moreover, the arguments used in the proof of Theorem \ref{t:rectr} also yields the existence of a weak solution to (\ref{e:ttforr}), which is unique in law. 

By direct manipulation one can also verify that $\tilde{s}$ and $\tilde{m}$ are a scale function and a speed measure for the solutions of (\ref{e:ttforr}). That $\tilde{P}^x(X_{\zeta}=r)=\tilde{s}(x)$ follows directly from the statement preceding (\ref{e:psiff}).

According to Theorem 5.5.29  in \cite{KS}  $\tilde{P}^x(\zeta=\infty)=1$ if and only if
\[
\lim_{x\rar r} \int_{y^*}^{x}\tilde{s}(x)-\tilde{s}(z)\tilde{m}(dz)=\lim_{x \rar l}\int_x^{y^*}\tilde{s}(z)-\tilde{s}(x)\tilde{m}(dz)=\infty.
\]
However, the above hold if and only if
\[
\int_{y^*}^{r}\frac{s(z)}{\sigma^2(z)s'(z)}dz=-\int_{l}^{y^*}\frac{s(z)}{\sigma^2(z)s'(z)}dz=\infty,
\]
which is equivalent to the martingale property of $s(X)$ by Theorem 1 of \cite{Kotani}.

In order to prove the remaining assertions let $l<a<b<r$ and $T_{a,b}:=\inf\{t:X_t \notin (a,b)\}$. Then, $N^{T_{a,b}}$ is a bounded positive martingale. Therefore,
\[
\tilde{P}^x(t <T_{a,b}, F)=\frac{E^x[\chf_{[t<T_{a,b}]}N_t]}{1+c|s(x)|}.
\]
Note that $T^{a,b}\rar \zeta$ under $P^x$ and $\tilde{P}^x$. However, $\zeta=\infty, \, P^x$-a.s. and (\ref{e:ttforrAC}) follows from the dominated convergence theorem.

(\ref{e:ttforrAC}) in particular implies
\[
\tilde{P}^x(\zeta>t)=\frac{E^x\left[ N_t\right]}{1+c|s(x)|}.
\]
Thus, $\tilde{P}^x(\zeta=\infty)=1$ iff 
\[
\lim_{t \rar \infty}\frac{E^x\left[ N_t\right]}{1+c|s(x)|}=1.
\]
However, since $N$ is a supermartingale, the above limit holds if and only if $N$ is a martingale. Hence, we conclude by the previous part, which has established the equivalence of the martingale property of $s(X)$ and $\tilde{P}^x(\zeta=\infty)=1$.
\end{proof}
\begin{example} Suppose that $X$ is a Brownian motion so that $y^*=0$. Then, taking $c=1$ in Proposition \ref{p:ttforr} implies that the transformed process is a weak solution of 
\[
X_t=x +B_t +\int_0^t \frac{\sgn(X_s)}{1+|X_s|}ds, \qquad t>0,
\]
where $\sgn(x)=-\chf_{[x<0]}+\chf_{[x\geq 0]}$. Roughly speaking $1+|X|$ behaves like  a $3$-dimensional Bessel process when $X$ is away from $0$. Observe that the above SDE has a non-exploding solution since Brownian motion is a martingale. Moreover, $X_{\infty}$ exists and equals $\infty$ or $-\infty$ with probabilities $\tilde{s}(x)$ and $1-\tilde{s}(x)$, respectively.
\end{example}
\section{Non-uniqueness of the Black-Scholes equation} \label{s:BS}
As promised earlier we will now apply the results of Sections \ref{s:rectr} and \ref{s:limit} to financial models, where the stock price movements are governed by a regular one-dimensional diffusion. To simplify the exposition we shall assume that the interest rate is $0$. Our interest is in the pricing equation for a derivative contract written on this stock. The Fundamental Theorem of Asset Pricing (see \cite{DS}) stipulates that the stock price must follow a local martingale under an equivalent probability measure, i.e. risk-neutral measure, and the price of the derivative contract equals the expectation of its terminal payoff under this measure if it is replicable.

Throughout this section we will assume that the stock price under the unique risk-neutral measure is given by
\be \label{e:bs}
X_t=X_0+\int_0^{t}\sigma(X_s)dB_s, \quad X_0>0,
\ee
on $(\Om, \cF,(\cF_t)_{t \geq 0}, \bbP)$, where $X_0$ is deterministic and $\sigma$ satisfies (\ref{e:ESR}) on $(0,\infty)$ as well as Assumption \ref{a:sholder}. In particular $X$ is the unique strong solution of the above equation. We also impose the condition that $X$ is a strict local martingale, i.e.
\be \label{e:LMc}
\int_1^{\infty}\frac{z}{\sigma^2(z)}dz<\infty.
\ee
\begin{remark} Note that we do not assume $X$ is always strictly positive, i.e. $X$ can hit $0$ in finite time with positive probability. 
\end{remark} 
The strict local martingale assumption places a {\em bubble} on the stock price in the sense that it   is valued  higher in the market than its expected future cash flows. Appearance of bubbles causes many standard results in derivative pricing theory become invalid (see \cite{CH} and \cite{PP}). In particular, the Cauchy problem associated to the prices of  European options do not admit a unique solution.
\begin{definition}
	Let $a>0$ and $b$ be measurable functions on $(0,\infty)$ and $D$ be an interval  in $[0,\infty)$. Consider a continuous function $g:D\mapsto \bbR$. A continuous function $u:[0,\infty)\times D\to \bbR$ is said to be a classical solution on $[0,\infty)\times D$ of 
	\bea
	u_t(t,x)&=&\frac{1}{2}a(x) u_{xx}(t,x)+ b(x)u_x(t,x)\label{e:defpde} \\
	u(0,x)&=&g(x), \label{e:defpdebc}
	\eea
	if $u \in C^{1,2}((0,\infty)\times \interior(D))$,  (\ref{e:defpde}) is satisfied for all $(t,x) \in (0,\infty)\times \interior(D)$ while (\ref{e:defpdebc}) is valid for all $x \in D$.
\end{definition} 
Given the above definition the following is an easy consequence of Theorem 3.2 in Ekstr\"om and Tysk \cite{ET}. For the rest of this section $D^*$ will denote $[0,\infty)$ if $P^x(\inf\{t: X_t=0\}<\infty)>0$ for some $x>0$. On the other hand, if $0$ is not accessible in finite time, $D^*:=(0,\infty)$.
\begin{theorem} \label{t:ET} Suppose that $\sigma$ satisfies (\ref{e:ESR}) on $(0,\infty)$, (\ref{e:LMc}) and Assumption \ref{a:sholder}. Consider a continuous  function $g:[0,\infty) \to [0,\infty)$ of at most linear growth and define on $[0,\infty)\times D^*$ the function $v(t,x):= E^x[g(X_t)]$, where $X$ is the unique solution of (\ref{e:bs}).  Assume further that $g(0)=0$ if $0\in D^*$. Then, $v$ is a classical solution on  $[0,\infty)\times D^*$ of the Cauchy problem
\be \label{e:cauchy}
v_t =\frac{1}{2}\sigma^2v_{xx}, \qquad v(0,\cdot)=g.
\ee
\end{theorem}
Non-uniqueness of the Cauchy problem is implicit in the above theorem. Indeed, if we let $g(x)=x$ and $w(x)=x$, both $w$ and $v$ are solutions of (\ref{e:cauchy}). Yet, $E^x[X_t]\neq x$ since $X$ is a strict local martingale.

The equation (\ref{e:cauchy}) is called the {\em Black-Scholes pricing equation} in the literature. If $g$ is the time-$T$ payoff of a European derivative written on the stock, $v(T-t,X_t)$ gives the time-$t$ price of this derivative, where $v$ is the solution of (\ref{e:cauchy}). On the other hand the arbitrage pricing theory states that the price of the derivative at time $t$ equals  $E^{X_t}[g(X_{T-t})]$ for a sufficiently well-behaved payoff since the risk-neutral measure is unique. Although Theorem \ref{t:ET} shows that the function defined by this alternative pricing formula still satisfies the Black-Scholes equation,  non-uniqueness of the Cauchy problem is problematic especially when one has to rely on numerical methods to find the price of the derivative.

The goal of this section is to identify {\em the stochastic solution}, $E^x[g(X_t)]$ in terms of  the unique solution of {\em some}
 Cauchy problem. The discussion following Theorem \ref{t:ET} shows that there is no hope if we work with the differential operator associated to the generator of $X$. However, the solutions of (\ref{e:limitRT}), which can be interpreted as the limit of recurrent transforms of $X$, or in view of Remark \ref{r:Yashtr} as an $h$-transform of $X$, come to our rescue.
\begin{theorem} \label{t:bsunique}
Suppose that $\sigma$ satisfies (\ref{e:ESR}) on $(0,\infty)$, (\ref{e:LMc}) and Assumption \ref{a:sholder}. Consider a continuous  function $g:[0,\infty) \to [0,\infty)$ of at most linear growth at infinity and $g(0)=0$ whenever $0\in D^*$. Let $v(t,x):= E^x[g(X_t)]$, where $X$ is the unique solution of (\ref{e:bs}),  for $(t,x )\in [0,\infty)\times D^*$. Then, the following statements are valid:
\begin{enumerate}
\item If $0\in D^*$, $v(t,0)=0$ for all $t \geq 0$.  
\item For $x>0$, $v(t,x)= x w(t,x)$, where $w$ is the unique  classical nonnegative solution on $[0,\infty)\times (0,\infty)$ of 
\bea
w_t(t,x)&=&\frac{1}{2}\sigma^2(x) w_{xx}(t,x) + \frac{\sigma^2(x)}{x}w_x(t,x) \label{e:wpde}\\
w(0,x)&=&\frac{g(x)}{x} \label{w:ic}
\eea
among the class of functions satisfying the following conditions:
\begin{enumerate}
	\item $w$ is of $O(x^{-1})$ as $x \rar 0$:
	\be \label{e:contat0+}
	\lim_{x \rar 0}\sup_{s \leq t}xw(s,x)<\infty, \qquad \forall t>0.
	\ee
	Moreover,  if $X$ reaches $0$ in finite time
	\be \label{e:contat0}
	\lim_{x \rar 0}\sup_{s \leq t}xw(s,x)=0, \qquad \forall t>0.
	\ee
	\item $w$ approaches to $0$ near infinity:
	\be
\forall t>0,\; \lim_{n \rar \infty}w(t_n,x_n)=0 \mbox{ if } x_n \uparrow \infty \mbox{ and }  t_n \rar t. \label{w:bc}
	\ee
\end{enumerate}
\item If $Y$ is  a weak solution of  (\ref{e:limitRT}) and $Q^x$ is its law, 
\be \label{e:w}
w(t,x)=Q^x\left[\frac{g(Y_t)}{Y_t}\chf_{[\zeta >t]}\right],
\ee
where $\zeta$ corresponds to the lifetime of $Y$.
\end{enumerate}
\end{theorem}
 Note that we do not require $\frac{g(x)}{x}$ to be bounded near $0$ in the above theorem. In particular, if $D^*=(0,\infty)$ and $g\equiv 1$, $w$ will be the solution of a Cauchy problem with the unbounded initial condition $\frac{1}{x}$. In this case the unique solution is given by $\frac{1}{x}=Q^x[\frac{1}{Y_t}]$ since $\frac{1}{Y}$ is a martingale when $X$ is strictly positive as observed in Proposition  \ref{p:Yinvmart}.
\begin{remark} In Theorem \ref{t:bsunique} the conditions (\ref{e:contat0+}) and (\ref{e:contat0})  are natural growth conditions near $0$ for the problem at hand given that we want $w$ satisfy $xw(t,x)=v(t,x)=E^x[g(X_t)]$. Indeed, $g(x) \leq K(1+x)$ implies $v(t,x)\leq K(1+x)$ since $X$ is a nonnegative local martingale, which in turn implies (\ref{e:contat0+}). Moreover, when $D^*=[0,\infty)$, $v$ will be uniformly continuous on $[0,t]\times[0,x]$ for all $x>0$ in view of the definition of a classical solution, which will lead to (\ref{e:contat0}).
	
  On the other hand, (\ref{w:bc}) must be imposed to achieve the intended uniqueness. Indeed, suppose that $X$ is a strictly positive strict local martingale and $g(x)=x$. Then both $1$ and $\frac{v(t,x)}{x}$ are classical solutions of (\ref{e:wpde}) with the initial condition (\ref{w:ic}) and satisfy the growth condition  (\ref{e:contat0+}). However, only $\frac{v(t,x)}{x}$ satisfies (\ref{w:bc}) as $\frac{v(t,x)}{x}=Q^x(\zeta>t)$. . 
\end{remark}
We end this section with the following immediate corollary to Theorem \ref{t:bsunique}, which implies that the function $x \mapsto E^x[X_t]$ is of strictly sublinear growth at infinity for  $t>0$.
\begin{corollary}
Suppose that $\sigma$ satisfies (\ref{e:ESR}) on $(0,\infty)$, (\ref{e:LMc}) and Assumption \ref{a:sholder}. Let $g$ be as in Theorem \ref{t:bsunique}. Then for every $t>0$
\[
\lim_{x \rar \infty}\frac{E^x[g(X_t)]}{x}=0,
\]
i.e. the function $x \mapsto E^x[g(X_t)]$ is of strictly sublinear growth at infinity. 
\end{corollary}
\section{Optimal stopping} \label{s:OS}
In this section we will consider the  following optimal stopping problem for a regular diffusion on $(l,r)$ satisfying Assumption \ref{a:reg}:
\be \label{e:OSP}
V(x):=\sup_{\tau \leq \zeta}E^x[e^{-\lambda \tau}g(X_{\tau})],
\ee
where $\lambda>0$ and $\tau$ is any stopping time with the usual convention that $e^{-\lambda \tau}g(X_{\tau}(\om))=\limsup_{t\rar \infty}e^{-\lambda t}g(X_{t}(\om))$ if $\tau(\om)=\zeta(\om)=\infty$. Here $g$ is taken to be a nonnegative function that is continuous on $I$.  From a financial perspective $V$ can be interpreted as the price of a perpetual American option with payoff $g$ on a stock whose dynamics are governed by $X$ and is currently priced at $x$  while $\lambda$ equals  the constant interest rate.  
\begin{remark} \label{r:cetal}
	This problem has been considered by Cisse et al. in \cite{CPT}, where the authors also use change of measure techniques with certain implicit assumptions on their way towards a solution. For instance, the proof of the key Lemma 3.5 is based on a result of Shiryaev, which requires the continuity of the function $f$ (of their Lemma 3.5) in the one-point compactifaction of the state space by adding the cemetery state. This in particular requires the boundedness of $g$ in (\ref{e:OSP}) with a certain behaviour at the boundary points. Moreover, the Sturm-Liuoville equation on p.1251 that defines the family of excessive functions $\phi^B(x)=E^x[e^{-qT_B}]$, where $T_B=\inf\{t>0:X_t \notin (a,b)\}$, stipulates that $\phi^B(a+)=\phi^B(b-)=1$, for all $a$ and $b$ satisfying $l\leq a<b\leq r$. However, this immediately rules out the case when $X$ has infinite lifetime or an entrance boundary. Indeed, if $a=l$ and $b=r$, $T_B=\zeta$. Thus, if the diffusion has infinite lifetime, $\phi^B(x)=0$ for all $x\in (l,r)$, which leads to $\phi^B(l+)=\phi^B(r-)=0$ by continuity. Similarly, if $l$ is an entrance boundary, $a=l$, and $b<r$, then $\phi^B(l+)=1$ implies $E^l[e^{-qT_B}]=1$, i.e. $P^l(T_B=0)=1$. This is a contradiction to the assumption that $l$ is an entrance boundary which entails that  the diffusion immediately enters the open interval $(l,r)$ right after time $0$ and never returns to $l$. Consequently, $P^l(T_B=T_b)=1$, where $T_b:=\inf\{t>0:X_t=b\}$. Clearly, $P^l(T_b>0)=1$.
		
	The method that is described below is applicable to all regular one-dimensional diffusions satisfying Assumption \ref{a:reg}. Aside from the above restrictions the method of Cisse et al. requires the knowledge of all $\phi^B$ for all open sets $B$. As we shall see later, our solution only requires  the knowledge of  $u^{\lambda}(\cdot,y)$ for some $y \in (l,r)$.
\end{remark}
To ease the exposition and simplify the proofs we shall assume from now on that  $X$ is on natural scale. We will solve the above problem   using mainly the $\lambda$-potential kernel, $u^{\lambda}$, and the recurrent transform introduced in Proposition \ref{p:AOtr}.  We start with the following lemma, which is a direct consequence of Proposition \ref{p:step1}.
\begin{lemma} \label{l:AO}
	Let $X$ be a regular diffusion satisfying Assumption \ref{a:reg} on $(l,r)$, $y \in (l,r)$ be fixed, and $g$ be a nonnegative measurable function on $I$.  Then, for any stopping time $\tau$ and $\lambda>0$, we have
	\be
	E^x\left[e^{-\lambda \tau}g(X_{\tau})\chf_{[\tau< \zeta]}\right]=
	u^{\lambda}(x,y)E^{h,x}\left[\frac{g(X_{\tau})}{u^{\lambda}(X_{\tau},y)}\exp\left(-\frac{L^y_{\tau}}{2u^{\lambda}(y,y)}\right)\chf_{[\tau <\infty]}\right],
	\ee
	where $E^{h,x}$ is the expectation with respect to $R^{h,x}$, which is the law of the recurrent transform in Proposition \ref{p:AOtr}.
\end{lemma} 

	Thus, the recurrent transform associated to $u^{\lambda}$ removes the discounting in the optimal stopping problem making it more tractable. We shall apply one more transformation to get rid of the local time factor in order to make the problem one-dimensional again. However, this recurrent transform will already give us the necessary condition for the finiteness of the optimal stopping problem in (\ref{e:OSP}) once we have the result from the next lemma. Throughout this section $E^{h,x}$ and $R^{h,x}$  will correspond to the expectation operator and the law  associated to  the solutions of (\ref{e:AOtr}), whose scale function can be chosen as follows for a given $y \in (l,r)$:
	\be \label{e:A0sr}
	s_h(x)=\int_y^x\frac{1}{(u^{\lambda}(z,y))^2}dz.
	\ee
	\begin{lemma} \label{l:LTab} For any $l<a <b <r$ and $x,y \in (a,b)$ we have
		\bea
		E^{h,x}\left[\chf_{[T_a<T_b]}\exp\left(-\frac{L^y_{T_a}}{2u^{\lambda}(y,y)}\right)\right]&=&\frac{R^{h,x}(T_a<T_b))}{1+s_a(b;y)s_h(b)^2u^{\lambda}(y,y)}, \mbox{ for } y \leq x;
		\\
		E^{h,x}\left[\chf_{[T_b<T_a]}\exp\left(- \frac{L^y_{T_b}}{2u^{\lambda}(y,y)}\right)\right]&=&\frac{R^{h,x}(T_b<T_a)}{1+s_b(a;y)s_h(a)^2u^{\lambda}(y,y)}, \mbox{ for } y \geq x
		\eea
		where
		\bean
		&&R^{h,x}(T_a<T_b)=\frac{s_h(b)-s_h(x)}{s_h(b)-s_h(a)}, \mbox{ and } \\
		&&s_a(b;x):=\int_a^x\frac{s_h'(z)}{(s_h(b)-s_h(z))^2}dz, \;s_b(a;x):=1-\int_x^b\frac{s_h'(z)}{(s_h(z)-s_h(a))^2}dz.
		\eean
	\end{lemma}
\begin{proof}
 Suppose $y \leq x$. Let us kill the recurrent transform as soon as it hits $a$ or $b$ and then apply an $h$-transform via 
\[
R^{h,x}(T_a<T_b)=\frac{s_h(b)-s_h(x)}{s_h(b)-s_h(a)}.
\]

This $h$-transform conditions the diffusion to converge to $a$. Thus, if we denote the law of this $h$-transform by $R^{h,a,x}$ and its potential kernel by ${u}_a$ (by dropping the dependence on $b$ to ease the notation), then
\begin{multline*}
E^{h,x}\left[\chf_{[T_a<T_b]}\exp\left(-c L^y_{T_a}\right)\right]=R^{h,x}(T_a<T_b)E^{h,a,x}\left[\exp\left(-c L^y_{\infty}\right)\right]\\
=R^{h,x}(T_a<T_b)E^{h,a,y}\left[\exp\left(-c L^y_{\infty}\right)\right]
=R^{h,x}(T_a<T_b)\frac{\frac{s_a'(b;y)}{2{u}_a(y,y)}}{c+\frac{s_a'(b;y)}{2{u}_a(y,y)}}
\end{multline*}
since $R^{h,a,x}(T_y<\infty)=1$ for $y \leq x$, $L^y_{\infty}$ is exponentially distributed under $R^{h,a,y}$ with parameter $\frac{s_a'(b;y)}{2{u}_a(y,y)}$, and $s_a$ is a scale function  of the above $h$-transform. Substituting $c$ with $ (2 u^{\lambda}(y,y))^{-1}$ and noticing $s_h(y)=0$, we arrive at  
\[
E^{h,x} \left[\chf_{[T_a<T_b]}\exp\left(- \frac{L^y_{T_a}}{2u^{\lambda}(y,y)}\right)\right]=\frac{s_h(b)-s_h(x)}{s_h(b)-s_h(a)}\frac{1}{1+s_a(b;y)s_h(b)^2u^{\lambda}(y,y)}.
\]

Similarly, for $y \geq x$,
	\[
	E^{h,x}\left[\chf_{[T_b<T_a]}\exp\left(- \frac{L^y_{T_b}}{2u^{\lambda}(y,y)}\right)\right]=\frac{s_h(x)-s_h(a)}{s_h(b)-s_h(a)}\frac{1}{1+s_b(a;y)s_h(a)^2u^{\lambda}(y,y)}.
	\]
	
	\end{proof}
\begin{proposition}\label{p:Vfinite}
	Let $x \in (l,r)$ be fixed and consider the value function, $V$, defined in (\ref{e:OSP}). If $V(x)$ is finite, then
	\be \label{e:NCopt}
	\liminf_{a \rar l}\frac{g(a)}{u^{\lambda}(a,x)s_h(a)}>-\infty \mbox{ and } \limsup_{b \rar r}\frac{g(b)}{u^{\lambda}(b,x)s_h(b)}<\infty.
	\ee
\end{proposition}
\begin{proof}
	Suppose that (\ref{e:NCopt}) is violated. Then, either $\liminf_{a \rar l}\frac{g(a)}{u^{\lambda}(a,x)s_h(a)}=-\infty$ or $\limsup_{b \rar r}\frac{g(b)}{u^{\lambda}(b,x)s_h(b)}=\infty$ or both. Suppose it is the former statement and, thus, there exists a sequnce $(a_n)$ with $a_n \rar l$ and 
		\be \label{e:explosion}
		\lim_{n \rar \infty}\frac{g(a_n)}{u^{\lambda}(a_n,x)s_h(a_n)}=-\infty.
		\ee
	Then, we claim that 
	\[
	\lim_{n \rar \infty} E^x\left[e^{-\lambda T_{n}}g(X_{T_n})\right]=\infty,
	\]
	where $T_n :=T_{a_n}\wedge \zeta$, which is in contradiction with  the hypothesis that $V(x)<\infty$. 
	
	Indeed, by Lemma \ref{l:AO}  and taking $y=x$, we have
	\bean
	E^x\left[e^{-\lambda T_{n}} g(X_{T_n})\chf_{[T_n<\zeta]}\right]&=&u^{\lambda}(x,x)E^{h,x}\left[\frac{g(X_{T_n})}{u^{\lambda}(X_{T_n},x)}\exp\left(-\frac{L^x_{T_n}}{2u^{\lambda}(x,x)}\right)\chf_{[T_n<\infty]}\right]\\
	&=&u^{\lambda}(x,x)\frac{g(a_n)}{u^{\lambda}(a_n,x)}E^{h,x}\left[\exp\left(-\frac{L^x_{T_{n}}}{2u^{\lambda}(x,x)}\right)\right],
	\eean
	where the last line is due to the recurrence of $X$ under $R^{h,x}$. However, Lemma \ref{l:LTab} together with the nonnegativity of $g$ now yield
	\[
	E^x\left[e^{-\lambda T_{n}} g(X_{T_n})\right]\geq u^{\lambda}(x,x)\frac{g(a_n)}{u^{\lambda}(a_n,x)}\lim_{b \rar r}\frac{1}{1+s_{a_n}(b;x)s_h^2(b)u^{\lambda}(x,x)}.
	\]
	On the other hand,
	\[
	\lim_{b \rar r}s_{a_n}(b;x)s_h^2(b)=\lim_{b \rar r}\int_{a_n}^x\frac{s_h'(z)s_h^2(b)}{(s_h(b)-s_h(z))^2}dz=\int_{a_n}^x\lim_{b \rar r}\frac{s_h'(z)s_h^2(b)}{(s_h(b)-s_h(z))^2}dz=-s_h(a_n)
	\]
	by the dominated convergence theorem. Recall that, since $x=y$, $s_h(x)=0$ by (\ref{e:A0sr}). Thus, the claim follows from (\ref{e:explosion}). 
	
	If, instead, $\limsup_{b \rar r}\frac{g(b)}{u^{\lambda}(b,x)s_h(b)}=\infty$, a similar construction  shows that $V(x)=\infty$ in that case, too. 
	\end{proof}
The above result shows that the boundedness of
\be \label{e:nc0}
z \mapsto \frac{g(z)}{u^{\lambda}(z,x)(1+|s_h(z)|)}
\ee
is necessary in order for $V(x)$ to be finite. In fact the condition (\ref{e:nc0}) is independent of $x$ and ensures $V(x)<\infty$ {\em for all} $x$, as one can also guess from the strong Markov property of $X$.
\begin{lemma} \label{l:nc}
	The mapping in (\ref{e:nc0}) is bounded if and only if for some $y \in (l,r)$
	\be \label{e:nc}
	z \mapsto \frac{g(z)}{u^{\lambda}(z,y)(1+|s_h(z)|)}
	\ee
	is bounded, where $s_h$ is defined by (\ref{e:A0sr}).
\end{lemma}
\begin{proof}
	It suffices to show that
	\[
	\sup_a\frac{u^{\lambda}(a,x)}{u^{\lambda}(a,y)}<\infty.
	\]

Indeed, by the symmetry  property of the potential kernels and (\ref{e:LThittime})
\[
\lim_{a \rar l}\frac{u^{\lambda}(a,x)}{u^{\lambda}(a,y)}=\lim_{a \rar l}\frac{u^{\lambda}(x,a)}{u^{\lambda}(y,a)}=\lim_{a \rar l}\frac{E^x[e^{-\lambda T_a}]}{E^y[e^{-\lambda T_a}]}.
\]
Moreover, if $x>y$, $E^x[e^{-\lambda T_a}]=E^y[e^{-\lambda T_a}]E^x[e^{-\lambda T_y}]$ by the strong Markov property. Thus, for $a<y<x$, $\lim_{a \rar l}\frac{u^{\lambda}(a,x)}{u^{\lambda}(a,y)}=E^x[e^{-\lambda T_y}]$.
Similarly, for $a<x<y$, $\lim_{a \rar l}\frac{u^{\lambda}(a,x)}{u^{\lambda}(a,y)}=\frac{1}{E^y[e^{-\lambda T_x}]}$.
The strong Markov property can be used also to show $\lim_{a \rar r}\frac{u^{\lambda}(a,x)}{u^{\lambda}(a,y)}<\infty$, concluding the proof.
\end{proof}
\begin{remark}
	A similar condition for the finiteness of the value function can be found in Part (I) of Theorem 6.3 in \cite{LZ}. Namely, the value function is finite if and only if
	\[
	\limsup_{x \rar l} \frac{g(x)}{\phi_{\alpha}(x)}<\infty  \mbox{ and } \limsup_{x \rar r} \frac{g(x)}{\psi_{\alpha}(x)}<\infty,
	\]
	where $\phi_{\alpha}$ and $\psi_{\alpha}$ are the fundamental solutions appearing in (\ref{e:ualpha}). On the other hand, (\ref{e:nc}) is equivalent to
	\[
	\limsup_{x \rar l} \frac{g(x)}{\psi_{\alpha}(x)|s_h(x)|}<\infty  \mbox{ and } \limsup_{x \rar r} \frac{g(x)}{\phi_{\alpha}(x)s_h(x)}<\infty.
	\]
	Combining the two conditions allows us to conclude that  $\frac{\phi^{\alpha}(x)}{\psi^{\alpha}(x)|s_h(x)|}$ (resp. $\frac{\psi^{\alpha}(x)}{\phi^{\alpha}(x)s_h(x)}$) remain bounded as $x \rar l$ (resp. $x \rar r$) when the above limits are nonzero.
\end{remark}
The above discussion justifies the following
\begin{assumption} \label{a:OS}
	For some (thus, for all) $y\in (l,r)$  the mapping in (\ref{e:nc}) is bounded.
\end{assumption}
The denominator in (\ref{e:nc}) should remind us of the transformation discussed in Section \ref{s:another}. Indeed, let us fix a $y \in (l,r)$ and remind ourselves that $(R^{h,x})_{x \in (l,r)}$ corresponds to the recurrent transform in Proposition \ref{p:AOtr} for $\alpha =\lambda$. Note that we can choose its scale function to be $s_h$ that is defined in (\ref{e:A0sr})  and satisfies $s_h(y)=0$.  The following follows immediately from Proposition \ref{p:ttforr} and Lemma \ref{l:AO}.
\begin{proposition} \label{p:AOtt} Suppose $X$ is a regular diffusion on natural scale satisfying Assumption \ref{a:reg} and let $c=\frac{u^{\lambda}(y,y)}{2}$. Then
	\begin{enumerate}
		\item For any $x \in (l,r)$ there exists a unique weak solution to 
		\be \label{e:AOtt}
		X_t= x+ \int_0^t \sigma(X_s)dB_s +\int_0^t \left\{\sigma^2(X_s)\frac{u^{\lambda}_x(X_s,y)}{u^{\lambda}(X_s,y)} - c\frac{s'_h(X_s)}{1-cs(X_s)}\chf_{[X_s\leq y]}+c\frac{s'_h(X_s)}{1+cs(X_s)}\chf_{[X_s> y]}\right\}, \quad t <\zeta,
		\ee
		where $\zeta:=\inf\{t: X_{t-}\in \{l,r\}\}$.
		\item The regular diffusion defined by (\ref{e:AOtt}) has scale function
		\be \label{e:AOttscale}
		\tilde{s}(x):=\frac{1+c(s_h(x)+|s_h(x)|)}{2(1+c|s_h(x)|)},
		\ee
		and speed measure
		\[
		\tilde{m}(dx)=\frac{4(1+c|s_h(x)|)^2}{c\sigma^2(x)s'_r(x)}dx.
		\]
		$\tilde{P}^x(X_{\zeta}=r)=\tilde{s}(x)=1-\tilde{P}^x(X_{\zeta}=l)$, where $\tilde{P}^x$ denotes the law of (\ref{e:ttforr}). 
		\item For any $F\in \cF_t$ the following absolute continuity relationship holds.
		\be \label{e:AOttAC}
		\tilde{P}^x(F, \zeta>t)=\frac{E^{h,x}\left[\chf_F \left(1+\frac{u^{\lambda}(y,y)}{2}|s_h(X_t)|\right)\exp\left(-\frac{L^y_t}{2 u^{\lambda}(y,y)}\right)\right]}{1+c|s_h(x)|}.
		\ee
		In particular, for any nonnegative continuous function $g$ on $I$ and stopping time $\tau$,
		\be \label{e:osid}
		E^x\left[e^{-\lambda \tau}g(X_{\tau})\chf_{[\tau< \zeta]}\right]=u^{\lambda}(x,y)\big(1+c|s_h(x)|\big)\tilde{E}^x\left[\frac{g(X_{\tau})}{u^{\lambda}(X_{\tau},y)\left(1+c|s_h(X_\tau)|\right)}\chf_{[\tau< \zeta]}\right].
		\ee
	\end{enumerate}
\end{proposition}
The identity (\ref{e:osid}) together with Assumption \ref{a:OS} allows us to solve (\ref{e:OSP}), which is the content of the next theorem whose proof is delegated to the Appendix.
\begin{theorem} \label{t:OS}
	Let $X$ be a regular diffusion on natural scale satisfying Assumption \ref{a:reg}. Consider  a nonnegative continuous function $g$  on $I$ satisfying Assumption \ref{a:OS}. Let $\tilde{s}$ be as in (\ref{e:AOttscale}) and $G$ be the smallest concave majorant on $(\tilde{s}(l),\tilde{s}(r))$ of the function
	\[
	\hat{g}(x):= \frac{g(\tilde{s}^{-1}(x))}{u^{\lambda}(\tilde{s}^{-1}(x),y)\left(1+\frac{u^{\lambda}(y,y)}{2}|s_h(\tilde{s}^{-1}(x))|\right)},
	\]
	and define
	\[
	\Gamma:=\{x \in (\tilde{s}(l),\tilde{s}(r)):\hat{g}(x)\geq G(x)\}.
\]
Then, 
\[
V(x)=u^{\lambda}(x,y)\Big(1+\frac{u^{\lambda}(y,y)}{2}|s_h(x)|\Big)G(\tilde{s}(x))<\infty.
\]
Moreover, the optimal stopping time for (\ref{e:OSP}) is
\[
\tau^*:=\inf\{t\geq 0:\tilde{s}(X_t) \in \Gamma\}.
\]
\end{theorem}
An immediate corollary to the above theorem is the following converse to the statement in Proposition \ref{p:Vfinite}. 
\begin{corollary} \label{c:Vfinite}
	Let $x \in (l,r)$ be fixed and consider the value function, $V$, defined in (\ref{e:OSP}). $V(x)$ is finite if and only if
	the mapping in (\ref{e:nc}) is bounded.
\end{corollary}
\begin{proof}
	The necessity has already been proved in Proposition \ref{p:Vfinite} in view of Lemma \ref{l:nc}.  Sufficiency follows from Theorem \ref{t:OS}. 
\end{proof}
\begin{remark}
	Note that the sole purpose of the assumption that $X$ is on natural scale in the above theorem is  to simplify the exposition. If $X$ is not on natural scale, then one can define $Y=s(X)$, which will be on natural scale, and consider instead the problem $\sup_{\tau}E^x[e^{-\lambda \tau}g(s^{-1}(Y_{\tau}))]$.
\end{remark}
\section{Conclusion} \label{s:conc}
We have introduced a new class of path transformations for one-dimensional regular diffusions aimed at modifying their behaviour towards recurrence. As a first application these transformations are used to compute the distribution of the first exit time from an interval for any diffusion. These transforms turned out to be instrumental in understanding strict local martingales better as well.  In Theorem \ref{t:bsunique} we give a novel characterisation of the Black-Scholes valuation formula in terms of the {\em unique} solution of an alternative Cauchy problem when the stock price is a local martingale and thus resolve the longstanding issue with the numerical computation of the option price when the option payoff is unbounded with linear growth. Finally, using the path transformations developed in this paper, we propose a unified framework for solving explicitly the optimal stopping problem for one-dimensional diffusions with discounting in Section \ref{s:OS}. Following Remark \ref{r:simulation} application of recurrent transformations to study the discrete Euler schemes for killed diffusion is left for future research.
\bibliographystyle{siam}
\bibliography{ref}
\appendix

\section{Proof of Theorem \ref{t:rectr}}
\begin{enumerate}[leftmargin=*]
		\item To show the first assertion it suffices to show that $h'$ equals a left-continuous function Lebesgue a.e. since the left derivative is  defined uniquely only outside  a Lebesgue null set. However, since $h'$ is assumed to be of finite variation, there exist non-decreasing functions $g^{+}$ and $g^{-}$ such that $h'=g^+ -g^-$. It follows from Exercise 12 in Chap. 7 of \cite{Rudin} that $g^+$ and $g^-$ are left-continuous a.e.. Thus, $h'$ is equal to a left-continuous function a.e..  
		
		Since $h'$ is of finite variation and can be taken to be left continuous, Exercise 13 in Chap. 7 of \cite{Rudin} shows that $h'$ can be viewed as a signed Borel measure on $(l,r)$. Then, it follows  from the Lebesgue decomposition theorem (Theorem C in Section 32 of \cite{Halmos}) and the Radon-Nikodym theorem (Theorem B in \cite{Halmos}) that the measure $dh'(x)$ admits the stated decomposition. That  $h''$ can be taken Borel measure follows from the fact that every Lebesgue measurable function is equal to a Borel measurable function a.e..
		
		\item Observe that, in view of occupation times formula, the integral in (\ref{e:integrability})  equals on $[t<\zeta]$
		\bean
		&&\int_l^r \left|\frac{\sigma^2(x)}{2}h''(x)+ b(x)h'(x)\right|\frac{L^x_{t }}{\sigma^2(x)}dx+\int_l^r \frac{L^x_{t }}{2}	|n(dx)|\\
		&=&\int_l^r \left|\half h''(x)+ \frac{b(x)}{\sigma^2(x)}h'(x)\right|L^x_{t }dx+\int_l^r \frac{L^x_{t }}{2}	|n(dx)|
		\eean
		Due to the continuity of $X$, on $[t<\zeta]$ and on almost every path $L^x_t$ would be equal to $0$ for all $x$ outside a compact interval in $(l,r)$, which is determined by the maximum and the minimum of $X$ on $[0,t]$. Thus, due to the continuity of $x \mapsto L^x_t$, it suffices to check 
		\be \label{e:MFv}
		\int_K \left|\half h''(x)+ \frac{b(x)}{\sigma^2(x)}h'(x)\right|dx+\int_K |n(dx)|<\infty
		\ee
		for an arbitrary compact $K$ contained in $(l,r)$. First note that
		\[
		\int_K \left|h''(x)\right|dx + \int_K |n(dx)|<\infty
		\]
		since $h'$ is of finite variation. 
		
		Moreover,
		\[
		\int_K \left|\frac{b(x)}{\sigma^2(x)}h'(x)\right|dx=C + \int_K \left(\int_c^y \left|\frac{2b(x)}{\sigma^2(x)}\right|dx\right)|dh'(y)|,
		\]
		for some $C<\infty$ and $c \in K$ due to the finiteness of $h'$ and $\int_c^y \left|\frac{2b(x)}{\sigma^2(x)}\right|dx$ at the boundary of $K$. However, the integral in the above representation is finite since $dh'$ is of finite variation and  $\int_c^y \left|\frac{2b(x)}{\sigma^2(x)}\right|dx$ is bounded in $K$. This completes the proof that (\ref{e:MFv}) holds for an arbitrary compact set $K$, which in turn yields the claim.
		\item It follows from the previous part that $\Lambda(h)$ is of finite variation. Since $(h(X_s)_{s \leq t}$ is away from $0$, path by path for $t<\zeta$, it immediately follows that $M$ is of finite variation, too.  
		
		Since $h$ can be considered as a difference of convex functions, it follows from It\^{o}-Tanaka formula that on $[t<\zeta]$
		\bean
		h(X_t)&=&h(x)+ \int_0^t h'(X_s)dX_s +\half\int_{(l,r)}L^x_t \left\{h''(x)dx+n(dx)\right\}\\
		&=&\int_0^t h'(X_s)dX_s +\half\int_0^t \sigma^2(X_s)h''(X_s)ds+ \half\int_l^r L^x_t n(dx).
		\eean
		Thus, a simple application of integration by parts formula yields
		\[
		h(X_t)M_t=h(x)+ \int_0^t h'(X_s)M_s\sigma(X_s)dB_s, \qquad t <\zeta,
		\]
		proving the local martingale property for $h(X)M$. In particular, $h(X)M$ is a continuous non-negative supermartingale with an integrable limit as $t \rar \zeta$. 
		
		Finally, due to the hypotheses on $s_h$ it follows from Theorem 5.5.15 in \cite{KS} that there exists a unique weak solution to (\ref{e:sderectr}). Moreover, the solution is recurrent by Part a) of Proposition 5.5.22 in \cite{KS}.
		\item Since $-s_h(l+)=s_h(r-)=\infty$, it follows that $h(l)=0$ (resp. $h(r)=0$) if $s(l)=0$ (resp. $s(r)=0$). That is, $h$ vanishes at the accessible boundaries and, thus, $h(X_{\zeta})=0$ on $[\zeta<\infty]$. Consequently,  $h(X_t)M_t=h(X_t)M_t\chf_{[t<\zeta]}$ since $M_t>0$ on $[t <\zeta]$ except on a $P^x$-null set by (\ref{e:integrability}). Moreover, $h$ must be strictly positive on $(l,r)$ in order for $s_h$ to be finite on $(l,r)$. These in turn yield the desired identity that $\inf\{t>0:h(X_t)M_t=0\}=\zeta, \,P^x$-a.s. \item In view of the previous part $\frac{h(X_t)}{h(x)}M_t$  is a supermartingale multiplicative functional satisfying Hypothesis 62.9 in \cite{GTMP}, that is, a supermartingale vanishing on $[\zeta,\infty)$. Then (\ref{e:ACrectr}) follows directly from Theorem 62.19 in \cite{GTMP} since $R^{h,x}(\zeta=\infty)=1$. Note that the space $\Om$ is projective in the terminology of Section 62 of \cite{GTMP} since it is the path space.  
		
		To show the martingale property observe that in view of (\ref{e:ACrectr}) and $R^{h,x}(\zeta=\infty)=1$,
		\[
		1=R^{h,x}(t<\zeta)=R^{h,x}(t<\infty)=\frac{1}{h(x)}E^x[h(X_t)M_t],
		\]
		yielding the martingale property of $h(X)M$ under $P^x$.
		\item  By the virtue of the monotone convergence theorem
		\[
		E^{h,x}\left[\frac{\chf_F}{h(X_T)M_T}\right]=\lim_{n \rar \infty}E^{h,x}\left[\chf_F\Big(\frac{1}{h(X_T)M_T}\wedge n\Big)\right].
		\]
		Thus, employing (\ref{e:ACrectr})  we arrive at
		\bean
		E^{h,x}\left[\frac{\chf_F}{h(X_T)M_T}\right]&=&\lim_{n \rar \infty} E^x\Big[\chf_F \big(\chf_{[h(X_T)M_T> \frac{1}{n}]}+ nh(X_T)M_T\chf_{[h(X_T)M_T\leq \frac{1}{n}]}\big)\Big]\\
		&=& P^x(F, h(X_T)M_T>0)+  \lim_{n \rar \infty} E^x\big[\chf_F nh(X_T)M_T\chf_{[T<\zeta]}\chf_{[h(X_T)M_T\leq \frac{1}{n}]}\big],
		\eean
		where the second line follows from the dominated convergence theorem and that $h(X_T)M_T=h(X_T)M_T\chf_{[T<\zeta]}, \,P^x$-a.s.. Moreover,
		\[
		E^x\big[\chf_F nh(X_T)M_T\chf_{[T<\zeta]}\chf_{[h(X_T)M_T\leq \frac{1}{n}]}\big]\leq P^x\Big(T<\zeta,h(X_T)M_T \leq \frac{1}{n}\Big),
		\]
		which converges to $0$ as $n \rar \infty$ since $h(X_T)M_T>0$ on $[T<\zeta]$ except on a $P^x$-null set by the previous part. Thus,
		\[
		E^{h,x}\left[\frac{\chf_F}{h(X_T)M_T}\right]=P^x(F, h(X_T)M_T>0)=P^x(F, T<\zeta).
		\]
		This completes the proof.
		
	\end{enumerate}

\section{Proof of Theorem \ref{t:rtrpot}}
	\begin{enumerate}[leftmargin=*]
		\item It follows from a simple differentiation of the potential functions in  (\ref{e:psiff})-(\ref{e:psiif}) that the left-derivative of $u(\cdot, y)$, i.e. $u_x(\cdot,y)$, at $x \in (l,r)$ is bounded by $s'(x)$ uniformly in $y$. Thus, since $\mu$ is a probability measure on $(l,r)$ and $s'$ is continuous under Assumption \ref{a:reg}, the dominated convergence theorem implies the left derivative of $h$ is given by
		\be \label{e:h'potential}
		h'(x)=\int_{(l,r)} u_x(x,y)\mu(dy), \qquad x\in (l,r).
		\ee
		Next, consider a finite subinterval $[a,b]$ of $(l,r)$ and recall that $u_x(\cdot,y)$ is non-decreasing on $(l,y)$ and non-increasing on $(y,r)$. Straightforward computation reveals that  the total variation of  $u_x(\cdot,y)$ on $[a,b]$, denoted by 	$\|u_x(\cdot,y)\|_{TV(a,b)}$, admits
		\[
		\|u_x(\cdot,y)\|_{TV(a,b)}\leq \sup_{x,z \in [a,b]}|s'(x)-s'(z)|\leq K(a,b)<\infty,
		\]
		for some constant $K(a,b)$ by the continuity of $s'$ and the compactness of $[a,b]$. Consequently,
		\[
		\|h'\|_{TV(a,b)}\leq \int_{(l,r)}\|u_x(\cdot,y)\|_{TV(a,b)}\mu(dy)\leq K(a,b),
		\]
		since $\mu((l,r))=1$. Thus, $h'$ is of finite variation. 
		
		Next, let $v(x,y):=u(s^{-1}(x), s^{-1}(y))$ and observe that $v(x,\cdot)$ is concave for each $x \in (l,r)$. Thus,
		\[
		h(s^{-1}(x))=\int_{(l,r)} v(x, s(y))\mu(dy)\leq v\left(x,\int_{(l,r)} s(y)\mu(dy)\right),
		\] 
		by Jensen's inequality. Next observe  that $s(l)<\int_{(l,r)} s(y)\mu(dy)<s(r)$.  Indeed, if $s(l)=-\infty$, $s(l)< \int_{(l,r)}s(y)\mu(dy)$ directly follows from the hypothesis that $\int_{(l,r)} |s(y)|\mu(dy)<\infty$.  If $s(l)=0$, since $s(x)\geq 0$ for all $x\geq l$, we have $\int_{(l,r)}s(y)\mu(dy)\geq 0$. In fact, $\int_{(l,r)}s(y)\mu(dy)>0$ since, otherwise, $s=0$, $\mu$-a.s.. However,  $\{x:s(x)=0\}=\{l\}$ as $s$ is strictly increasing under Assumption \ref{a:reg}. Thus, $\int_{(l,r)} s(y)\mu(dy)>0$ due to the hypothesis that $\mu$ does not charge $\{l\}$. Similarly, we can show $\int_{(l,r)}s(y)\mu(dy)<s(r)$.  Moreover, by the continuity of $s$, there exists  some $y^* \in I$  such that $\int_{(l,r)} s(y)\mu(dy)=s(y^*)$. 
		
		Therefore, $h(x) \leq v(s(x), s(y^*))=u(x,y^*)$. This in turn implies
		\be \label{e:shlbl}
		\int_l^{y^*}\frac{s'(y)}{h^2(y)}dy\geq |s_u(l+)|,
		\ee
		 where
		\[
		s_u(x)=\int_{y^*}^x \frac{s'(z)}{(u(z,y^*))^2}dz, \qquad x\in (l,r).
		\]
		Suppose, first, that $s(l)=1-s(r)=0$. Then, for $x <y^*$,
		\[
		s_u(x)=\frac{1}{(1-s(y^*))^2}\int_{y^*}^x\frac{s'(z)}{s^2(z)}dz=\frac{1}{(1-s(y^*))^2}\left(\frac{1}{s(y^*)}-\frac{1}{s(x)}\right),
		\]
		which in particular shows that $\lim_{x \rar l}s_u(x)=-\infty$. Similarly, for $x >y^*$,
		\[
		s_u(x)=\frac{1}{s^2(y^*))}\int_{y^*}^x\frac{s'(z)}{(1-s(z))^2}dz=\frac{1}{s^2(y^*)}\left(\frac{1}{1-s(x)}-\frac{1}{1-s(y^*)}\right),
		\]
		and, thus, $s_u(r-)=\infty$. The other cases are handled  the same way to show $-s_u(l+)=s_u(r-)=\infty$. This in turn yields in view of (\ref{e:shlbl}) that $s_h(l+)=-\int_l^{y^*} \frac{s'(y)}{h^2(y)}dy=-\infty$. Similarly, $s_h(r-)\geq s_u(r-)=\infty$.
		
		Thus, $h$ satisfies the conditions of Theorem \ref{t:rectr}. In particular,  $h'$ can be taken left-continuous with the Lebesgue decomposition $dh'(x)=h''(x)dx+ n(dx)$, where $n$ is a locally finite signed measure that is singular with respect to the Lebesgue measure. Moreover, $(h,M)$ is a recurrent transform where
		\[
		M_t=\exp\left(-\int_0^t \frac{\tbA h(X_s)}{h(X_s)}ds -\int_0^t \frac{1}{h(X_s)}d\Lambda_s(h)\right),
		\]
		where $\Lambda_t(h)=\int_{(l,r)} \frac{L^x_t}{2}n(dx)$. 
		
		On the other hand,  the occupation times formula applied to $\int_0^t \frac{\tbA h(X_s)}{h(X_s)}ds$ yields 
		\[
		M_t=\exp\bigg(-\int_{(l,r)} \frac{L^x_t}{h(x)}\Big(\half dh'(x)+ \frac{b(x)h'(x)}{\sigma^2(x)}dx\Big)\bigg).
		\] 
		Thus, we will be done if $\half dh'(x)+ \frac{b(x)h'(x)}{\sigma^2(x)}dx=-\half s'(x)\mu(dx)$  on $(l,r)$.
		Note that this will follow if  for any continuous $f$ with a compact support in $(l,r)$, we establish
		\be \label{e:dhipot}
		\half  \int_{(l,r)}f(x)dh'(x)+ \int_l^r f(x)\frac{b(x)h'(x)}{\sigma^2(x)}dx=-\half\int_{(l,r)}f(x)s'(x)\mu(dx).
		\ee
		First note that $u_x(\cdot,y)$ is differentiable everywhere except at $y$ under Assumption \ref{a:reg}. Using this observation and the fact that the jump in the (left-continuous) left-derivative $u_x(x,y)$ at $x=y$ equals $u_x(y+,y)-u_x(y,y)=-s'(y)$, we deduce
		\[
		du_x(x,y)=u_{xx}(x,y)dx - s'(y)\eps_y(dx),
		\]
		for some function $u_{xx}(\cdot,y)$ that is a.e. uniquely determined  by the second derivative of $u(\cdot,y)$, which exists  at each $x \neq y$. Moreover, 
		\be \label{e:tbAu}
		\half \sigma^2(x)u_{xx}(x,y)+ b(x)u_x(x,y)=0, \qquad \forall x \neq y.
		\ee
		Next, the second integral on the left hand side of (\ref{e:dhipot}) equals
		\bean
		&&\int_l^r f(x)\frac{b(x)}{\sigma^2(x)}\left(\int_{(l,r)}u_x(x,y)\mu(dy)\right)dx= \int_{(l,r)}\mu(dy)\int_l^r f(x)\frac{b(x)}{\sigma^2(x)}u_x(x,y)dx\\
		&=& -\half\int_{(l,r)}\mu(dy)\left(\int_l^yf(x)u_{xx}(x,y)dx+\int_y^r f(x)u_{xx}(x,y)\right)\\
		&=&-\half\int_{(l,r)}\mu(dy)\left(f(y)(u_x(y,y)-u_x(y+,y))-\int_l^rf'(x)u_{x}(x,y)dx\right)\\
		&=&-\half \left(\int_{(l,r)}s'(y)f(y)\mu(dy) -\int_l^r f'(x)\int_{(l,r)}u_x(x,y)\mu(dy)dx\right),
		\eean
		where the first equality follows from Fubini's theorem since $|u_x(x,y)|\leq s'(x)$ as observed before and $f$ has compact support. The second line above is a consequence of (\ref{e:tbAu}) and the third line is a straightforward integration by parts. The last line is a consequence of $u_x(y+,y)-u_x(y,y)=-s'(y)$ and another application of Fubini's theorem due to the aforementioned bound on $u_x$. Since $\int_l^r f'(x)\int_{(l,r)}u_x(x,y)\mu(dy)dx=\int_l^r f'(x)h'(x)dx$, (\ref{e:dhipot}) follows from a simple integration by parts. 
		\item This is a restatement of the final part of Theorem \ref{t:rectr} in this special case.
	\end{enumerate}

\section{Proof of Theorem \ref{t:prctr}}
	\begin{enumerate}[leftmargin=*]
		\item As in the proof of Theorem \ref{t:rtrpot}, one can differentiate from the left under the integral sign since $\int_{(l,r)} u^{\alpha}(y,y)\mu(dy)<\infty$ and
		\[
		u^{\alpha}_x(x,y)\leq \left(\frac{\psi_{\alpha}'(x)}{\psi_{\alpha}(x)}+\frac{\phi_{\alpha}'(x)}{\phi_{\alpha}(x)}\right)u^{\alpha}(y,y),
		\]
		where $\psi_{\alpha}$ and $\phi_{\alpha}$ are the fundamental solutions as in (\ref{e:ualpha}). The fact that $h'$ is of finite variation can be shown similarly using the representation of (\ref{e:ualpha}) and the continuity properties of the fundamental solutions.
		
		If $X$ is transient, the potential kernel $u$ exists and we have
		\[
		h(x)\leq  \int_{(l,r)} u(x,y)\mu(dy).
		\]
		Thus, $-s_h(l+)=s_h(r-)=\infty$ by Theorem \ref{t:rtrpot}.
		
		Also note that  $u^{\alpha}_x(x,y)$ is differentiable from left at all $x \neq y$ with the left-derivative $u_{xx}^{\alpha}(x,y)$ satisfying $\half \sigma^2 u^{\alpha}_{xx}(x,y)+ b(x)u^{\alpha}_x(x,y)=\alpha u^{\alpha}(x,y)$ for $x \neq y$ (see Paragraphs 10 and 11 in Section II.1 of \cite{BorSal}).
		Moreover, $u^{\alpha}_x(y+,y)-u^{\alpha}_x(y,y)=\frac{\phi'_{\alpha}(y)\psi_{\alpha}(y)-\psi'_{\alpha}(y)\phi_{\alpha}(y)}{w_{\alpha}}=-s'(y)$. 
		Thus, $du^{\alpha}_x(x,y)=u^{\alpha}_{xx}(x,y)dx-s'(y)\eps_y(dx)$, and the same arguments  of the proof of Theorem \ref{t:rtrpot}  can be used to show that $(h,M)$ is a recurrent transform. 
		
		Now, suppose $X$ is recurrent. By applying a scale transformation we may assume without loss of generality that $X$ is on natural scale. This in turn implies $-l=r=\infty$. Using the fact that $u^{\alpha}(x,y)\leq u^{\alpha}(y,y)$ 
		\[
		h(x)\leq \int_{(-\infty,\infty)}u^{\alpha}(y,y)\mu(dy)<\infty,
		\]
		we deduce $\int_c^{\infty}\frac{1}{h^2(x)}dx\geq \int_c^{\infty}\frac{1}{ (\int_{-\infty}^{\infty}u^{\alpha}(y,y)\mu(dy))^2}dx=\infty$. That is, $s_h(r-)=\infty$. Similarly, $s_h(l+)=-\infty$ and that $(h,M)$ is a recurrent transform follows again from the same lines of the proof of Theorem  \ref{t:rtrpot} in view of the aforementioned properties of $u^{\alpha}$.
		
		\item Note that the speed measure of the recurrent transform is given by $h^2(x)m(dx)$. Thus, we need to show that the speed measure is finite 	since the stationary distribution of a one-dimensional diffusion is given by its speed measure when it is finite (see p.37 of \cite{BorSal}).
		
		Using Jensen's inequality and Fubini's theorem we get
		\[
		\int_l^r h^2(x)m(dx)\leq \int_l^r \int_{(l,r)} (u^{\alpha}(x,y))^2 \mu(dy)m(dx)=\int_{(l,r)} \int_l^r (u^{\alpha}(x,y))^2 m(dx)\mu(dy).
		\]
		Moreover,
		\bean
		\int_l^r(u^{\alpha}(x,y))^2m(dx)&=&\int_l^r \int_0^{\infty}\int_0^{\infty}e^{-\alpha(t+s)}p(t,x,y)p(s,x,y)dsdtm(dx)\\&=& \int_0^{\infty}\int_0^{\infty}e^{-\alpha(t+s)}\int_l^r p(t,y,x)p(s,x,y)m(dx)dsdt\\
		&=&\int_0^{\infty}\int_0^{\infty}e^{-\alpha(t+s)}p(t+s,y,y)dsdt\\
		&=&\int_0^{\infty}ue^{-\alpha u} p(u,y,y)du\leq \frac{1}{\eps}\int_0^{\infty}e^{-(\alpha-\eps) u} p(u,y,y)du\\
		&=&\frac{u^{\alpha-\eps}(y,y)}{\eps}.
		\eean
		In above, the first equality follows from (\ref{e:ualpha}), the second is due to the symmetry of the transition density and the Fubini's theorem, while the third is a consequence of Chapman-Kolmogorov identity.
		
		Therefore, $\int_l^r h^2(x)m(dx) \leq \int_{(l,r)}\frac{u^{\alpha-\eps}(y,y)}{\eps}\mu(dy)<\infty$.
	\end{enumerate}

\section{Proof of Theorem \ref{t:bsunique}}
	If $0 \in D^*$, and $X_0=0$, $X_t=0$ for all $t>0$ $P^0$-a.s. since $0$ is an absorbing boundary. Thus, $v(t,0)=E^0[g(X_t)]=E^0[g(0)]=0$ since $g(0)=0$ when $0\in D^*$.
	
	As mentioned in Remark \ref{r:Yashtr}, the process $Y$ can be considered as an $h$-transform of $X$ with $h(x)=x$. Indeed, if $\tau_n:=\inf\{t:X_t \notin (\frac{1}{n},n)\}$, $X^{\tau_n}$ is a bounded martingale and a straightforward application of Girsanov's theorem yields that $X$ is a weak solution of (\ref{e:limitRT}) up to $\tau_n$. Therefore,
	\[
	E^x\left[g(X_t)\chf_{[t<\tau_n]}\right]=E^x\left[\chf_{[t<\tau_n]}\frac{g(X_t)}{X_t}X_t\right]=xQ^x\left[\chf_{[t<\tau_n]}\frac{g(X_t)}{X_t}\right],
	\]
	where $Q^x$ is the unique law of solutions of (\ref{e:limitRT}). Observe that $\tau_n$ converges to the lifetime, $\zeta$, of $X$ under $P^x$ and $Q^x$. Moreover, since $X$ is a positive supermartingale, $P^x(\zeta=\inf\{t:X_t=0\})=1$ while $Q^x(\zeta=\inf\{t:X_{t}=\infty\})=1$ by Proposition \ref{p:Xvshtr} since the scale function of (\ref{e:limitRT}) is $1-1/x$. 
	Thus, the monotone convergence theorem   together with the assumption that   $g(0)=0$ when $0$ is an accessible boundary under $P^x$ yields 
	\be \label{e:vrep}
	xQ^x\left[\chf_{[t<\zeta]}\frac{g(X_t)}{X_t}\right]=E^x\left[g(X_t)\chf_{[t<\zeta]}\right]=E^x\left[g(X_t)\right]=v(t,x).
	\ee
	Thus, $v(t,x)=xw(t,x)$, where $w$ is as defined in (\ref{e:w}), since the law of $Y$ is the same as that of  $X$ under $Q^x$. 
	
	Recall from Theorem \ref{t:ET}  that $v$ satisfies (\ref{e:cauchy}). This automatically implies $w$ satisfy (\ref{e:wpde}) and (\ref{w:ic}). To prove the other properties for $w$ fix an $m>0$ and note that
	\bea
	w(t,x)&=&Q^x\left[\chf_{[t<\zeta]}\chf_{[Y_t>m]}\frac{g(Y_t)}{Y_t}\right]+ Q^x\left[\chf_{[t<\zeta]}\chf_{[Y_t<m]}\frac{g(Y_t)}{Y_t}\right]\nn \\
	&\leq& K_1(m) Q^x(\zeta>t) + K_2(m)Q^x\left[\frac{1}{Y_t}\right]\nn \\
	&\leq& K_1(m) Q^x(\zeta>t) + \frac{K_2(m)}{x}\label{e:wubd},
	\eea
	where the second line follows since $\frac{g(y)}{y}$ is bounded on $(m,\infty)$ by the linear growth assumption and $g$ is continuous on $[0,m]$, and the third line is due to $\frac{1}{Y}$ being supermartingale since $1-1/x$ is a scale function for $Y$.  Note that the above immediately yields (\ref{e:contat0+}).
	
	Let us next  show that $w$ satisfies (\ref{e:contat0}) when $0$ is an accessible boundary under $P^x$. Indeed, $xw(s,x)=E^x[g(X_s)]$ implies that $(s,x)\mapsto xw(s,x)$ is jointly continuous due to the Feller property of $X$ and the continuity of paths\footnote{Feller property implies in particular that $(P^x)_{x \geq 0}$ is a continuous family of laws.}.  Moreover, $g(0)=0$ if $0$ is an accessible (and, therefore, absorbing) boundary by our hypothesis. Thus, $\lim_{x\rar 0} E^x[g(X_s)]=g(0)=0$ and the convergence is uniform on compact intervals, yielding  (\ref{e:contat0}).
	
	To show that $w$ also satisfies (\ref{w:bc}) it suffices to show in view of (\ref{e:wubd})  that
	\[
	\lim_{n \rar \infty} Q^{x_n}(\zeta>t_n)=0,
	\]
	which will hold if $Q^x(\zeta>t)$ tends to $0$ as $x \rar \infty$ for an arbitrary but fixed $t>0$. Indeed, since $Q^x(\zeta>t)$ is decreasing in $t$, we have, for any monotone sequence $(t_n)_{n \geq 1}$ with limit $s$, 
	\[
	\lim_{n \rar \infty} Q^{x_n}(\zeta>t_n) \leq \lim_{n \rar \infty} Q^{x_n}(\zeta>s \wedge t_1)=0.
	\]
	Motivated by the above define $w_0(t,x):=Q^x(\zeta>t)$ and pick $y>x$. Then
	\bean
	Q^x(\zeta>t)&=&Q^x(\zeta>t, T_y<t)+Q^x(\zeta>t, T_y>t)\\
	&=&E^x\left[\chf_{[T_y<t]}w_0(t-T_y,y)\right]+Q^x(T_y>t),
	\eean 
	by the strong Markov property of $Y$. Taking limits as $y \rar \infty$ we obtain
	\[
	Q^x(\zeta>t)=E^x\left[\chf_{[\zeta<t]}\lim_{y \rar \infty}w_0(t-T_y,y)\right]+Q^x(\zeta>t)
	\]
	since $T_y \rar \zeta$, $Q^x$-a.s. as $y \rar \infty$. Observe that the interchange of expectation and limit is justified by bounded convergence. However, the above readily implies that $\lim_{y \rar \infty}w_0(t-T_y,y)=0$ on $[\zeta <t]$. Note that $T_y<\zeta$ and $w_0(s,y)$ is decreasing in $s$ for fixed $y$. Thus, $w_0(t,y)\leq w_0(t-T_y,y)$ and we deduce that  $\lim_{y \rar \infty}w_0(t,y)=0$ since $Q^x(\zeta <t)>0$ for all $x>0$.
	
	Now, conversely assume that $w$ is a classical solution of (\ref{e:wpde})-(\ref{w:ic}) satisfying (\ref{e:contat0+})-(\ref{w:bc}). We shall  show that $w$ satisfies (\ref{e:w}). First note that (\ref{w:bc}) implies $w$ is bounded near infinity:
	\be \label{w:bdatinfty}
	\sup_{y>x, s\in [r,t]}w(r,y)<\infty, \;  \forall x>0 \mbox{ and } 0\leq r <t <\infty.
	\ee
	
	Next define $\tau_{n,m}:=\inf\{t:Y_t \notin (\frac{1}{m},n)\}=T_n \wedge T_{\frac{1}{m}}$, where $Y$ is a weak solution of (\ref{e:limitRT}), and observe that $\lim_{n \rar \infty}\lim_{m \rar \infty}\tau_{n,m}=\zeta$, $Q^x$-a.s. for every $x>0$. Fix a $T>0$ and let $f(t,x):=w(T-t,x)$. Employing It\^{o}'s formula and using the continuity of $\sigma$ and $w_x$ we get
	\[
	w(T,x)= Q^x\left[f(T\wedge \tau_{n,m}, Y_{T \wedge \tau_{n,m}})\right]=Q^x\left[\frac{g(Y_T)}{Y_T}\chf_{[T<\tau_{n,m}]}\right]+Q^x\left[w(T-\tau_{n,m}, Y_{\tau_{n,m}})\chf_{[T >\tau_{n,m}]}\right]
	\]
	since $Q^x(T=\tau_{n,m})=0$.
	
	Note that the first term of the summation converges to 
	\[
	Q^x\left[\frac{g(Y_T)}{Y_T}\chf_{[T<\zeta]}\right]
	\]
	by monotone convergence. 
	
	Let us first suppose that $0$ is an accessible boundary for $X$. Then
	\bean
	Q^x\left[w(T-\tau_{n,m}, Y_{\tau_{n,m}})\chf_{[T\geq \tau_{n,m}]}\right]&=&Q^x\left[w(T-T_{\frac{1}{m}}, \frac{1}{m})\chf_{[T\geq \tau_{n,m}]}\chf_{[T_n>T_{\frac{1}{m}}]}\right]\\
	&&+Q^x\left[w(T-T_n, n)\chf_{[T\geq \tau_{n,m}]}\chf_{[T_n<T_{1/m}]}\right]\\
	&\leq& K(m) m Q^x\left(T_n>T_{\frac{1}{m}}\right) +Q^x\left[w(T-T_n, n)\chf_{[T\geq \tau_{n,m}]}\chf_{[T_n<T_{\frac{1}{m}}]}\right],
	\eean
	where $K(m)$ is a constant satisfying $\lim_{m \rar \infty}K(m)=0$ due to (\ref{e:contat0}).  On the other hand, 
	\[
	\lim_{m \rar \infty} m Q^x\left(T_n>T_{\frac{1}{m}}\right)=m \frac{\frac{1}{x}-\frac{1}{n}}{m-\frac{1}{n}}=1.
	\]
	Thus,
	\[
	\lim_{m \rar \infty}Q^x\left[w(T-\tau_{n,m}, Y_{\tau_{n,m}})\chf_{[T\geq \tau_{n,m}]}\right]\leq Q^x\left[w(T-T_n, n)\chf_{[T\geq \tau_{n,\infty}]}\chf_{[T_n<T_{0}]}\right]
	\]
	by the dominated  convergence theorem. Recall that $T_0=\infty$ as $Y$ does not hit $0$. Thus,
	\[
	\lim_{n \rar \infty}\lim_{m \rar \infty}Q^x\left[w(T-\tau_{n,m}, Y_{\tau_{n,m}})\chf_{[T\geq \tau_{n,m}]}\right]\leq Q^x\left[\lim_{n \rar \infty}w(T-T_n, n)\chf_{[T\geq \zeta]}\right]
	\]
	by the dominated convergence theorem in view of (\ref{w:bdatinfty}) and the fact that  $T_n \rar \zeta$, $Q^x$-a.s..  Moreover, $T-\tau_n$ is decreasing to $T-\zeta$ on the set $T\geq \zeta$. Thus, it follows from  (\ref{w:bc}) that the above limit is $0$  since $Q^x(\zeta=T)=0$.  Hence, 
	\[
	w(t,x)=Q^x\left[\frac{g(Y_T)}{Y_T}\chf_{[T<\zeta]}\right]
	\]
	This completes the proof of uniqueness when $0$ is an accessible boundary for $X$ under $P^x$.
	
	To complete the proof let us now assume that $0$ is inaccessible, i.e. $X$ is strictly positive. This means that $\frac{1}{Y}$ is a true martingale under $Q^x$ in view of Proposition \ref{p:Yinvmart}.  The proof will be complete once we show again that 
	\[
	\lim_{n \rar \infty}\lim_{m \rar \infty}Q^x\left[w(T-\tau_{n,m}, Y_{\tau_{n,m}})\chf_{[T\geq \tau_{n,m}]}\right]=0.
	\]
	Indeed,
	\bea
	Q^x\left[w(T-\tau_{n,m}, Y_{\tau_{n,m}})\chf_{[T\geq \tau_{n,m}]}\right]&=&Q^x\left[w(T-T_{\frac{1}{m}}, \frac{1}{m})\chf_{[T\geq \tau_{n,m}]}\chf_{[T_n>T_{\frac{1}{m}}]}\right]\nn\\
	&&+Q^x\left[w(T-T_n, n)\chf_{[T\geq\tau_{n,m}]}\chf_{[T_n<T_{1/m}]}\right]\nn\\
	&\leq& K Q^x\left[\frac{1}{Y_{T_{\frac{1}{m}}}}\chf_{[T\geq \tau_{n,m}]}\chf_{[T_n>T_{\frac{1}{m}}]}\right]\nn\\
	&&+Q^x\left[w(T-T_n, n)\chf_{[T\geq\tau_{n,m}]}\chf_{[T_n<T_{1/m}]}\right]\nn\\
	&=&K Q^x\left[\frac{1}{Y_T}\chf_{[T\geq \tau_{n,m}]}\chf_{[T_n>T_{\frac{1}{m}}]}\right] \label{e:est1}\\
	&&+Q^x\left[w(T-T_n, n)\chf_{[T\geq\tau_{n,m}]}\chf_{[T_n<T_{1/m}]}\right] \label{e:est2},
	\eea
	where the first inequality is due to (\ref{e:contat0+}) and the second equality is a consequence of the martingale property of $\frac{1}{Y}$.
	
	Convergence of (\ref{e:est2}) to $0$ as $m$ and $n$ diverge to $\infty$ can be shown as before. Thus, it remains to show that (\ref{e:est1}) converges to $0$ as well. 
	
	Indeed, by the dominated  convergence theorem
	\[
	\lim_{m \rar \infty} Q^x\left[\frac{1}{Y_T}\chf_{[T\geq \tau_{n,m}]}\chf_{[T_n>T_{\frac{1}{m}}]}\right] = Q^x\left[\frac{1}{Y_T}\lim_{m \rar \infty}\chf_{[T\geq \tau_{n,m}]}\chf_{[T_n>T_{\frac{1}{m}}]}\right] =0
	\]
	for sufficiently large $n$ since $\lim_{m \rar \infty}\chf_{[T_n>T_{\frac{1}{m}}]}=\chf_{[T_n=\infty]}$ and $T^n<\infty$, $Q^x$-a.s. for any $n>x$ due to the fact that $Y$ is strictly positive and explodes in finite time, $Q^x$-a.s. by Proposition \ref{p:Xvshtr}.
	\section{Proof of Theorem \ref{t:OS}}
	\begin{itemize}
		\item[Step 1:] Let's first see that if $a$ is an accessible boundary under $P^x$ for some $x \in (l,r)$, then $\lim_{z \rar a} \frac{g(z)}{u^{\lambda}(z,y)(1+|s_h(z)|)}<\infty$. Indeed, since $g$ is continuous on $I$ and $a\in I$, $g(a)<\infty$. On the other hand, $\lim_{z\rar a}u^{\lambda}(z,y)=0$ (see Table 1 in \cite{MKpar}). Thus, a straightforward application of L'Hospital's rule yields $\lim_{z\rar a}u^{\lambda}(z,y)(1+|s_h(z)|)=\lim_{z\rar a}\frac{1}{u^{\lambda}_x(z,y)}>0$ since $a$ is a regular absorbing boundary (see, again, Table 1 in \cite{MKpar}). This in particular implies $\hat{g}(\tilde{s}(a))=G(\tilde{s}(a))$.
		\item[Step 2:] Note that $G$ is well-defined and bounded due to Assumption \ref{a:OS}. Let $Y=\tilde{s}(X)$ and observe that $Y$ is a local martingale under $\tilde{P}^x$ and $G$ is the least excessive majorant of $\hat{g}$ on $D:=(\tilde{s}(l),\tilde{s}(r))$ for $Y$. Note that by the continuity of $G$ we can extend it continuously to $\tilde{s}(\tilde{I})$, where $\tilde{I}$ is the state space of $X$ under $\tilde{P}_x$ for any $x \in (l,r)$. Moreover $G$ will be the smallest concave majorant  on $\tilde{s}(\tilde{I})$  of $\hat{g}(x)\chf_{x\in D}(x)$, which is lower semicontinuous on $\tilde{s}(\tilde{I})$.  Therefore, for $x \in (l,r)$
		\[
		\sup_{\tau}\tilde{E}^x[\hat{g}(\tilde{s}(X_{\tau}))\chf_{[\tau<\zeta]}]=G(\tilde{s}(x))
		\]
		by Theorem 1 on p. 124 of \cite{Shiryaev} or Proposition 3.3 in \cite{DK}. Moreover, Lemma 8 and Theorem 2 in Chapter 3 of \cite{Shiryaev} establish for any $\eps >0$ that  $\tilde{P}^x(\tau_{\eps}^*<\zeta)=1$ and $\tilde{E}^x[G(\tilde{s}(X_{\tau_{\eps}^*}))]=G(\tilde{s}(x))$, where 
		\[
		\tau_{\eps}^*:=\inf\{t\geq 0: \hat{g}(\tilde{s}(X_t))+\eps \geq G(\tilde{s}(X_t)) \}.
		\]
		\item[Step 3:] Let $v(x):=u^{\lambda}(x,y)\Big(1+\frac{u^{\lambda}(y,y)}{2}|s_h(x)|\Big)G(\tilde{s}(x))$ for $x \in I$. Observe that $v(a)$ is well-defined by the continuity whenever $a$ is an accessible boundary in view of Step 1 and that $v(x)\geq g(x)$ for all $x \in I$ by construction. Moreover, $
		\big(e^{-\lambda t}v(X_t)\big)$ is a $P^x$-supermartingale.  Also observe that $v(a)=g(a)$ whenever $a$ is an accessible boundary due to the construction of $G$ and Step 1. Thus, if $P^x(\zeta<\infty)=1$,  $\hat{g}(\tilde{s}(X_{\zeta}))=G(\tilde{s}(X_{\zeta})$ implying  $P^x(\tau_{\eps}^*<\zeta)=1$ by the continuity of $\hat{g}$ and $G$. On the other hand,  if $P^x(\zeta<\infty)=0$,
		\[
		P^x(\tau_{\eps}^*>t)=\frac{v(x)}{G(\tilde{s}(x))}\tilde{E}^x\Big[\frac{1}{{u^{\lambda}(X_{\tau_{\eps}^*},y)\left(1+c|s_h(X_{\tau_{\eps}^*})|\right)}}\chf_{[\tau_{\eps}^*>t]}\Big],
		\] 
		which converges to $0$ as $t \rar \infty$ by dominated convergence since $\tilde{P}^x(\tau_{\eps}^*<\zeta)=1$.  Thus, in view of Step 2, we obtain
		\be \label{e:epsopt}
		E^x[e^{-\lambda \tau_{\eps}^*}v(X_{\tau_{\eps}^*})]=\frac{v(x)}{G(\tilde{s}(x))}\tilde{E}^x[G(\tilde{s}(X_{\tau_{\eps}^*})]=v(x).
		\ee
		\item[Step 4:] The above in fact implies $E^x[e^{-\lambda \tau^*}v(X_{\tau^*})]=v(x)$. Indeed, letting $\rho_n:=\inf\{t\geq0: e^{-\lambda t}v(X_t)\geq n\}$, we have in view of (\ref{e:epsopt}) that
		\[
		v(x)=E^x[e^{-\lambda \tau_{\eps}^*}v(X_{\tau_{\eps}^*})]\leq E^x[e^{-\lambda \rho_n\wedge\tau_{\eps}^*}v(X_{\rho_n \wedge \tau_{\eps}^*})]
		\]
		since $\big(e^{-\lambda t}v(X_t)\big)$  is  a $P^x$-supermartingale.
		Thus, by  virtue of the dominated convergence theorem as $\eps\rar 0$, we deduce $v(x)\leq E^x[e^{-\lambda \rho_n\wedge\tau^*}v(X_{\rho_n \wedge \tau^*})]$. However, $E^x[e^{-\lambda \rho_n\wedge\tau^*}v(X_{\rho_n \wedge \tau^*})]$ converges to $E^x[e^{-\lambda \tau^*}v(X_{ \tau^*})]$ by the monotone convergence theorem as $n \rar \infty$. This shows $v(x)\leq E^x[e^{-\lambda \tau^*}v(X_{ \tau^*})]$, which in turn yields the claim as $v(x)\geq E^x[e^{-\lambda \tau^*}v(X_{ \tau^*})]$ by the supermartingale property of $\big(e^{-\lambda t}v(X_t)\big)$ and  Fatou's lemma.
		\item[Step 5:] We will now see that $V=v$. To this end let $\tau$ be an arbitrary stopping time and $\tau_n:=\inf\{t\geq 0: X_t \notin (a_n,b_n)\}$, where $l<a_n<b_n<r$ are such that $a_n \rar l$ and $b_n \rar r$ as $n \rar \infty$. Then, by Fatou's lemma  we obtain $E^x[e^{-\lambda \tau}g(X_{\tau})]\leq E^x[e^{-\lambda \tau\wedge \tau_n}g(X_{\tau\wedge \tau_n})]=\frac{v(x)}{G(\tilde{s}(x))}\tilde{E}^x\left[\tilde{g}(\tilde{s}(X_{\tau\wedge \tau_n}))\right]\leq v(x)$  in view of Step 2.
		
		On the other hand, 
		\be \label{e:epsoptV}
		E^x[e^{-\lambda\tau^*_{\eps}}g(X_{\tau^*_{\eps}})]=\frac{v(x)\tilde{E}^x[\tilde{g}(\tilde{s}(X_{\tau^*_{\eps}}))]}{G(\tilde{s}(x))}\geq \frac{v(x)\tilde{E}^x[G(\tilde{s}(X_{\tau^*_{\eps}}))-\eps]}{G(\tilde{s}(x))}=\frac{v(x)(G(\tilde{s}(x)-\eps)}{G(\tilde{s}(x))}.
		\ee
		By letting $\eps \rar 0$, we arrive at $V(x)\geq v(x)$, which in turn implies $V=v$. The fact that $V$ is finite is a consequence of Assumption \ref{a:OS}.
		\item[Step 6:] It remains to show that $\tau^*$ is optimal. Indeed, since $e^{-\lambda \tau_{\eps}^*}v(X_{\tau_{\eps}^*})$ converges to $e^{-\lambda \tau^*}v(X_{\tau^*})$ in $L^1(P^x)$ as observed in Step 4, and $v\geq g$, $(e^{-\lambda \tau_{\eps}^*}g(X_{\tau_{\eps}^*}))_{\eps >0}$ is a uniformly integrable family. Therefore, with the help of  (\ref{e:epsoptV}), we arrive at
		\[
		E^x\left[e^{-\lambda \tau^*}g(X_{\tau^*})\right]=\lim_{\eps \rar 0}E^x\left[e^{-\lambda \tau_{\eps}^*}g(X_{\tau_{\eps}^*})\right]\geq \lim_{\eps \rar 0} \frac{v(x)(G(\tilde{s}(x)-\eps)}{G(\tilde{s}(x))}=v(x).
		\]
	\end{itemize}
\end{document}